\documentclass[a4paper,10pt,reqno]{article}
\usepackage[export]{adjustbox}
\usepackage{silence}
\usepackage{blindtext}
\usepackage[skip=1pt plus1pt, indent=15pt]{parskip}
\usepackage{fullpage}
\usepackage{enumitem}
\usepackage[normalem]{ulem}
\usepackage[table]{xcolor} 
\usepackage{multirow}
\usepackage{mathtools}
\usepackage{amsmath}
\usepackage{adjustbox}
\usepackage{amsfonts}
\usepackage{amssymb}
\usepackage{multirow}
\usepackage{array}
\usepackage{makeidx}
\usepackage{xspace}
\usepackage{esvect}
\usepackage[most]{tcolorbox}
\usepackage{cite}
\usepackage{xspace} 
\usepackage{float}
\usepackage{tikz}
\usepackage{placeins}
\usepackage{babel}
\usepackage{graphicx}
\usepackage[caption=false]{subfig} 

\makeatletter
\renewcommand{\p@subfigure}{\thefigure}
\makeatother
\usepackage{colortbl} 
\definecolor{lightgray}{rgb}{0.9,0.9,0.9} %
\usepackage[left=1.3cm,right=1.3cm,top=1.5cm,bottom=1.5cm]{geometry}
\usepackage{amsthm,amssymb,mathrsfs, pstricks,booktabs}
\allowdisplaybreaks
\newtheorem{thm}{Theorem}[section]
\newtheorem{theorem}{Theorem}[section]
\newtheorem{lem}{Lemma}

\usepackage{latexsym}
\usepackage{graphicx}
\usepackage{epstopdf}
\usepackage{play}
\usepackage{epsfig}
\usepackage[nottoc]{tocbibind}

\usepackage[utf8]{inputenc}
\usepackage[version=4]{mhchem}
\usepackage{stmaryrd}
\usepackage{bbold}

\usepackage{sectsty}
\sectionfont{\centering}
\makeatletter
\renewenvironment{proof}[1][\proofname]{%
  \par\pushQED{\qed}\normalfont%
  \topsep6\p@\@plus6\p@\relax
\trivlist\item[\hskip\labelsep\bfseries#1\@addpunct{.}]%
  \ignorespaces
}{%
  \popQED\endtrivlist\@endpefalse
}
\makeatother
\usepackage{blindtext}

\usepackage{hyperref}
\hypersetup{
    linkcolor=black,
    filecolor=black,      
    urlcolor=black,
    pdfpagemode=FullScreen,}
\begin{document}
\begin{center}
    \textbf{\large SOLITONIC AND EXACT SOLUTIONS FOR A VISCOUS TRAFFIC FLOW MODEL VIA LIE SYMMETRY}
\end{center}
\vspace{0.5cm}
\begin{center}
\author[{\small \textbf{Urvashi Joshi}\textsuperscript{1}, \textbf{Aniruddha Kumar Sharma}\textsuperscript{2}, \textbf{Rajan Arora}\textsuperscript{3}\\ \textsuperscript{1,2,3} \text{Department of Applied Mathematics and Scientific Computing,}\\ \text{Indian Institute of Technology, Roorkee, India}}\\
\end{center}
\vspace{0.2cm}
\begin{center}
\textsuperscript{1}urvashi\textunderscore j@amsc.iitr.ac.in,
\textsuperscript{2}aniruddha\textunderscore s@as.iitr.ac.in, \\\textsuperscript{3}rajan.arora@as.iitr.ac.in.\textbf{(Corresponding Author)}
\end{center}
\section*{Abstract}
\qquad This work studies a macroscopic traffic flow model driven by a system of nonlinear hyperbolic partial differential equations. Using Lie symmetry analysis, we determine the infinitesimal generators and construct an optimal system of one-dimensional subalgebras, facilitating symmetry reductions for the governing system. In addition, we discussed the classical symmetry and solution of the traffic flow model with the initial conditions left invariant. By applying the method of nonlinear self-adjointness, conservation laws associated with the model are established and are utilized to obtain exact solutions. Using these exact solutions, we construct solitonic solutions, including kink-type, peakon-type, and parabolic solitons. Additionally, using the weak discontinuity C¹ wave illustrates nonlinear wave dynamics in traffic evolution. Moreover, we investigate how these solutions affect traffic behavior, clarifying shock wave development and flow stability. The results provide a basis for useful applications in traffic management, real-time traffic control, and intelligent transportation systems, as well as improving mathematical knowledge of traffic dynamics.

\textbf{Keywords:} Traffic flow model, Lie symmetry, Invariant initial condition, Exact and soliton solutions, Conservation laws, Self-adjointness, Weak discontinuities.  

\section{Introduction}
\qquad Mathematical modeling using partial differential equations (PDEs) has become an essential tool in science and engineering fields, allowing researchers to characterize and examine complex structures prevalent in materials science, quantum physics, and fluid dynamics. In transportation science, PDE-based models play a pivotal role in understanding and optimizing traffic flow, addressing pressing challenges like congestion, queue formation, and strategic management of road infrastructure. The demand for strong and scalable traffic flow models that guide urban planning and real-time management continues to increase as urbanization accelerates and transportation networks become more complicated.\par

Traffic flow modeling aims to capture the collective dynamics of vehicles on roadways, providing critical information to reduce congestion, design infrastructure, and assess environmental impact \cite{horvat2015traffic}. These models are generally classified into microscopic, mesoscopic, and macroscopic frameworks. While macroscopic models \cite{mohan2013state} treat traffic as a continuum, using fluid dynamics analogies to efficiently simulate large-scale traffic patterns and events such as shock waves and queue dissipation, microscopic models \cite{kanagaraj2013evaluation} focus on individual vehicle behaviors but are computationally intensive, mesoscopic models \cite{van2015genealogy,ferrara2018microscopic} use probabilistic approaches to balance detail and efficiency.
Among macroscopic models, the Aw-Rascle-Zhang (ARZ) model \cite{aw2000resurrection,aw2002derivation} stands out as a higher-order framework that extends the classical Lighthill-Whitham-Richards (LWR) model \cite{aw2002derivation,richards1956shock,lighthill1955kinematic} by incorporating both vehicle density and velocity, as well as driver relaxation and pressure effects. This makes the ARZ model extremely relevant for both theoretical studies and practical uses in real-time traffic management. It enables it to more accurately depict complicated traffic events, including momentum conservation, shock wave propagation, and stop-and-go oscillations.  

Recent studies \cite{paliathanasis2022lie,paliathanasis2023symmetry} have utilized Lie symmetry analysis to derive invariant solutions for traffic models, yet most focus on idealized or inviscid environments, neglecting realistic effects like viscosity and adaptation of driver behavior. Moreover, limited attention has been paid to solitonic traffic waves and the role of initial condition-preserving symmetries. This study addresses these gaps by applying Lie symmetry methods to a viscous macroscopic model, constructing exact solutions, analyzing weak discontinuities, and exploring solitonic dynamics, thus improving both the theoretical and practical understanding of traffic evolution.

The practical utility of macroscopic models is evident in their widespread adoption for city-scale traffic simulation, infrastructure evaluation, and the optimization of control strategies such as signal timings and ramp metering. These models are particularly valued for their computational efficiency, which does not deteriorate with increasing network size or vehicle count, and their ability to provide aggregate measures like average travel times, fuel consumption, and emissions.

In this work, we study a hydrodynamic macroscopic traffic flow model governed by a non-linear hyperbolic PDE system that incorporates vehicle density, average velocity, and traffic pressure. Reflecting the flexibility of drivers to adapt to changing traffic conditions, the model comprises a continuity equation for vehicle conservation and a momentum equation with relaxation and pressure factors. We apply Lie symmetry analysis \cite{bluman2010applications,olver1993applications} to determine infinitesimal generators and construct an optimal system of one-dimensional subalgebras, enabling systematic symmetry reductions and the derivation of exact invariant solutions. Direct approaches are used in obtaining conservation laws; therefore, the property of nonlinear self-adjointness \cite{hussain2024lie} is investigated to guarantee the physical significance of these laws. An in-depth understanding of nonlinear wave dynamics and the evolution of traffic patterns is given by using traveling wave transformations to build solitonic solutions, including kink, peakon, and parabola solitons.

This work develops theoretical knowledge of traffic flow dynamics and a useful toolbox accessible for congestion management and optimization of transportation systems by combining advanced mathematical methodologies with real-world modeling requirements. The results affect the design of intelligent transportation systems, the evolution of real-time traffic control methods, and the general implementation of applied mathematics in traffic engineering.
 
\subsection{Governing Model}
\qquad We consider a macroscopic traffic flow model described by a system of non-linear hyperbolic partial differential equations \cite{khoudari2024microscopic,shagolshem2024exact}. This hydrodynamic model treats traffic as a one-dimensional compressible fluid, where the key state variables are vehicle density and average velocity. The model consists of two primary equations:
\begin{itemize}
    \item \textbf{Continuity Equation:} This equation expresses the conservation of vehicles \cite{shagolshem2024exact} and is given by:
\end{itemize}
\begin{equation}\label{EC}
    \frac{\partial \rho}{\partial t}+\frac{\partial (\rho u)}{\partial x}=0,
\end{equation}
where $\rho(x, t)$ represents the vehicle density at position $x$ and time $t$, and $u(x, t)$ is the average velocity of vehicles at the same location and time.
\begin{itemize}
    \item \textbf{Equation of Motion:} This equation describes the dynamics of the average velocity, accounting for pressure gradients, and relaxation towards an equilibrium velocity \cite{paliathanasis2023symmetry, shagolshem2024exact}:
\end{itemize}
\begin{equation}\label{EM}
    \frac{\partial u}{\partial t}+u \frac{\partial u}{\partial x}=-\frac{1}{\rho}\frac{\partial P}{\partial x}
    +\frac{R}{\tau}(u_{e}(\rho)-u),
\end{equation}
where $P(x, t)$ represents the traffic pressure, $u_e(\rho)$ is the equilibrium velocity, $\tau$ is the characteristic relaxation time, and $R$ is a dimensionless relaxation coefficient. The relaxation term $\frac{R}{\tau}(u_{e}(\rho)-u)$ models the tendency of drivers to adjust their speed towards the equilibrium velocity $u_e(\rho)$ over time.
The traffic pressure, $P$, is defined in analogy with fluid mechanics as:
\begin{equation}
P = A\rho - D\frac{\partial u}{\partial x},
\end{equation}
where $A$ represents the speed variance, and $D$ represents viscosity, both considered constant parameters. Substitution of this expression for $P$ into the equation of motion yields:
\begin{equation}\label{EM2}
    \frac{\partial u}{\partial t}+u \frac{\partial u}{\partial x}+\frac{A}{\rho}\frac{\partial \rho}{\partial x}-\frac{D}{\rho}\frac{\partial^2 u}{\partial x^2}=
    \frac{R}{\tau}(u_{e}(\rho)-u).
\end{equation}

In the homogeneous case, we consider $R = 0$ or $u_e(\rho) = u$, representing an equilibrium situation where the traffic is in steady state.

The system of equations can be concisely written as:
\begin{equation}\label{TM}
 \begin{split}
&\mathbb{R}_{1}: \rho u_{x}+\rho_{x} u+ \rho_{t}=0\text{,}\\
&\mathbb{R}_{2}:  u_{t}+u u_{x}+\frac{A \rho_{x}}{\rho}-\frac{D u_{xx}}{\rho}=0\text{.}
\end{split}
\end{equation}

The model presented above is a viscous macroscopic model. When $D=0$, it simplifies to a non-viscous model. This macroscopic technique lets one analyze total traffic behavior by considering traffic flow as a compressible fluid, therefore offering an understanding of events, including shock wave generation and congestion patterns.
\subsection{Outline}
\qquad This research article is structured as follows. Section 2 defines the Lie symmetry analysis, the initial condition left invariant by classical symmetries, and the formulation of the optimal system. Section 3 addresses similarity reductions and invariant solutions. Section 4 examines the findings and their physical explanations. Section 5 analyzes non-linear self-adjointness and formulates conservation laws. Section 6 examines the progression of weak discontinuities. Section 7 obtains exact solutions via conservation laws. Section 8 proposed the application of this model. Finally, Section 9 concludes the study with key findings and future directions.
\section{Lie Symmetry and Optimal Subalgebras}
\subsection{Lie Symmetry Analysis}
\qquad We analyze a one-parameter Lie group of transformations with parameter $\varepsilon$ (where $\varepsilon << 1$) for the system \eqref{TM}.
\begin{equation}
\begin{split}
&x^{\star}=x+\varepsilon \mathscr{F}_{x}(x, t, \rho, u)+O\left(\varepsilon^{2}\right)\text{,}\\
&t^{\star}=t+\varepsilon \mathscr{F}_{t}(x, t, \rho, u)+O\left(\varepsilon^{2}\right)\text{,}\\
&\rho^{\star}=\rho+\varepsilon \mathscr{F}_{\rho}(x, t, \rho, u)+O\left(\varepsilon^{2}\right)\text{,}  \\
&u^{\star}=u+\varepsilon \mathscr{F}_{u}(x, t, \rho, u)+O\left(\varepsilon^{2}\right)\text{.} \end{split}
\end{equation}
Here, $\mathscr{F}_{x}$, $\mathscr{F}_{t}$, $\mathscr{F}_{\rho}$, and $\mathscr{F}_{u}$ represent the infinitesimals corresponding to the transformations of the independent variables $x$, $t$, and the dependent variables $\rho$, $u$, respectively.

\begin{theorem}
If we propose Lie point symmetry generators for the system \ref{TM} in the following form:
\begin{equation}
\mathbb{P} = \mathscr{F}_x(x,t,\rho,u)\frac{\partial}{\partial x} + \mathscr{F}_t(x,t,\rho,u)\frac{\partial}{\partial y} + \mathscr{F}_\rho(x,t,\rho,u)\frac{\partial}{\partial p} + \mathscr{F}_u(x,t,\rho,u)\frac{\partial}{\partial u},
\end{equation}
we have $\mathbb{P}$ satisfy the condition
\begin{equation}\label{PR}
    \operatorname{Pr}^{(2)}\mathbb{P}(\mathbb{R}_i)\big|_{\mathbb{R}_i=0} = 0, \hspace{0.1cm} \text{for} \hspace{0.1cm}  i=1,2,
\end{equation}
where $\operatorname{Pr}^{(2)}\mathbb{P}$ is the second prolongation of $\mathbb{P}$, $\mathbb{R}_{1}=\rho u_{x}+\rho_{x} u+ \rho_{t}=0$, and 
$\mathbb{R}_{2}=u_{t}+u u_{x}+\frac{A \rho_{x}}{\rho}-\frac{D u_{xx}}{\rho}$.
\end{theorem}
\begin{proof}
We set $\mathcal{X} = (x, t)$, $\mathcal{U} = (\rho, u)$, so that $\mathbb{R}_{1}=\rho u_{x}+\rho_{x} u+ \rho_{t}=0$, and 
$\mathbb{R}_{2}=u_{t}+u u_{x}+\frac{A \rho_{x}}{\rho}-\frac{D u_{xx}}{\rho}$
is defined in $\mathcal{X} \times \mathcal{U}^{(2)}$. Because the highest derivative of system \ref{TM} is second-order, we need the second prolongation of $\mathbb{P}$. Meanwhile, $\mathbb{R}_i$ is defined in $\mathcal{X} \times \mathcal{U}^{(2)}$, so we have $\mathbb{P}$ satisfy the condition $\operatorname{Pr}^{(2)}\mathbb{P}(\mathbb{R}_i)\big|_{\mathbb{R}_i=0} = 0$ for $i=1,2$.

Expanding Eq. \ref{PR} and equating the coefficients of all the partial derivatives of $\rho$ and $u$, we can obtain the following.
\begin{equation}\label{DE}
(\mathscr{F}_t)_{tt} = (\mathscr{F}_x)_{tt} = (\mathscr{F}_t)_\rho = (\mathscr{F}_t)_u = (\mathscr{F}_t)_x = (\mathscr{F}_x)_\rho = (\mathscr{F}_x)_u = 0,\hspace{0.2cm} (\mathscr{F}_x)_{x} = (\mathscr{F}_t)_{t}, \hspace{0.2cm}
(\mathscr{F}_\rho) = -\rho (\mathscr{F}_t)_t,\hspace{0.2cm}
(\mathscr{F}_u)= (\mathscr{F}_x)_t.
\end{equation}
By solving Eqs. \ref{DE}, we derive the expressions of $\mathscr{F}_x, \mathscr{F}_t, \mathscr{F}_\rho, \mathscr{F}_u $ as
\begin{equation}
\mathscr{F}_{x}=e_1 x +e_3 t+ e_4, \quad \mathscr{F}_{t}=e_1 t +e_2, \quad \mathscr{F}_{\rho}= -e_1 \rho, \quad \mathscr{F}_{u}=e_3
\end{equation}
where $e_1, e_2, e_3, e_4$ are the real constants.
\end{proof}
Using the application of Lie group analysis\cite{bluman2010applications,olver1993applications}, we determine the Lie point symmetry generators of system \ref{TM} are presented in Table \ref{S1}.
\renewcommand{\arraystretch}{2.0}
\begin{table}[H]
    \centering
    \begin{tabular}{cccc}
    \hline \hline
      \textbf{Conditions} & \textbf{Symmetry Vectors} & \textbf{Symmetries} \\
        \hline \hline
   $e_1 \neq 0$   & $\mathscr{S}_{1}= x \frac{\partial}{\partial x}+t \frac{\partial}{\partial t}-\rho \frac{\partial}{\partial \rho}$ & \text{Dilation in $x$, $t$, $\rho$}\\
  $e_2 \neq 0$ & $\mathscr{S}_{2}= \frac{\partial}{\partial t}$& \text{Translation in $t$}\\
        $e_3 \neq 0$ & $\mathscr{S}_{3}= t \frac{\partial}{\partial x}+\frac{\partial}{\partial u}$ & \text{Galilean boost}\\
        $e_4 \neq 0$ & $\mathscr{S}_{4}= \frac{\partial}{\partial x}$ & \text{Translation in $x$}\\
   \hline \hline
    \end{tabular}
     \caption{Infinitesimal Generators and Symmetries.}
    \label{S1}
\end{table}
The Lie bracket is defined as $[\mathscr{S}_i,\mathscr{S}_j]=\mathscr{S}_{i}\mathscr{S}_{j}-\mathscr{S}_{j}\mathscr{S}_{i}$ and the set $\{\mathscr{S}_1, \mathscr{S}_2,\mathscr{S}_3,\mathscr{S}_4\}$ is closed under this Lie bracket.
In Table \ref{Table1} we show the commutation relations between the Lie algebra $\{\mathscr{S}_1, \mathscr{S}_2,\mathscr{S}_3,\mathscr{S}_4\}$. Also, from Table \ref{Table1}, we conclude that the above algebra is solvable.
\renewcommand{\arraystretch}{2.0}
\begin{table}[H]
    \centering
    \begin{minipage}{0.48\textwidth}
        \centering
        \begin{tabular}{c|cccc}
        \hline \hline
            $\star$ & $\mathscr{S}_1$ & $\mathscr{S}_2$ & $\mathscr{S}_3$ & $\mathscr{S}_4$ \\ \hline \hline
            $\mathscr{S}_1$ & 0 & $-\mathscr{S}_2$ & 0 & $-\mathscr{S}_4$  \\ 
            $\mathscr{S}_2$ & $\mathscr{S}_2$ & 0 & $\mathscr{S}_4$ & 0\\ 
            $\mathscr{S}_3$ & 0 & $-\mathscr{S}_4$ & 0 & 0\\ 
            $\mathscr{S}_4$ & $\mathscr{S}_4$ & 0 & 0& 0\\ \hline \hline
        \end{tabular}
        \caption{Commutation Table.}
        \label{Table1}
    \end{minipage}
\hspace{-1cm} 
    \begin{minipage}{0.48\textwidth}
        \centering
        \begin{tabular}{c|cccc}
        \hline \hline
            $\operatorname{Ad}(\exp{\varepsilon(\star)}\star)$ & $\mathscr{S}_1$ & $\mathscr{S}_2$ & $\mathscr{S}_3$ & $\mathscr{S}_4$\\ \hline \hline
            $\mathscr{S}_1$ & $\mathscr{S}_1$ & $e^{\varepsilon}\mathscr{S}_2$ & $\mathscr{S}_3$ & $e^{\varepsilon} \mathscr{S}_4$ \\ 
            $\mathscr{S}_2$ & $\mathscr{S}_1-\varepsilon \mathscr{S}_2$ & $\mathscr{S}_2$ & $\mathscr{S}_3-\varepsilon \mathscr{S}_{4}$ & $\mathscr{S}_4$ \\ 
            $\mathscr{S}_3$ & $\mathscr{S}_1$ & $\mathscr{S}_2+\varepsilon \mathscr{S}_{4}$ & $\mathscr{S}_3$ & $\mathscr{S}_4$\\ 
            $\mathscr{S}_4$ & $\mathscr{S}_1-\varepsilon \mathscr{S}_4$ & $\mathscr{S}_2$ & $\mathscr{S}_3$ & $\mathscr{S}_4$  \\ \hline \hline
        \end{tabular}
        \caption{Adjoint Representation Table.}
        \label{AD}
    \end{minipage}
\end{table}
\subsection{Initial Condition Left-Invariant by Classical Symmetries}
\qquad Using the method of classical symmetric analysis with initial condition left-invariant \cite{goard2014symmetries,al2012nonclassical,zhang2010classical}, in the traffic flow model (\ref{TM}) is equivalent to 
\begin{equation}\label{IC2}
 \begin{split}
& \rho_{t}=-\rho u_{x}-\rho_{x} u\text{,}\\
& u_{t}=-u u_{x}-\frac{A \rho_{x}}{\rho}+\frac{D u_{xx}}{\rho}\text{,}
\end{split}
\end{equation}
and satisfies the initial conditions
\begin{equation}
 \begin{split}
& \rho(t,x)|_{t=0}=\Theta_{1}(x)\text{,}\\
& u(t,x)|_{t=0}=\Theta_{2}(x)\text{.}
\end{split}
\end{equation}
The operator under classical Lie symmetry can be written as 
\begin{equation}
    \mathbb{P}= (e_1 x +e_3 t+ e_4) \frac{\partial}{\partial_{x}}+ (e_1 t +e_2)\frac{\partial}{\partial_{t}}-(e_1 \rho) \frac{\partial}{\partial_{\rho}}+e_3 \frac{\partial}{\partial_{u}}\text{.}    \end{equation}
    The operator $\mathbb{P}$ is invariant when it satisfies the initial condition left invariant in \cite{goard2014symmetries}. We have 
    \begin{equation}
        \mathbb{P}(t-0)|_{t=0}=0 \Rightarrow e_2=0,
    \end{equation}
    and 
    \begin{equation}\label{IC1}
        \begin{split}
            & \mathbb{P}(\rho-\Theta_{1}(x))|_{\rho=\Theta_{1}(x), t=0}=0 \text{,}\\
             & \mathbb{P}(u-\Theta_{2}(x))|_{u=\Theta_{2}(x), t=0}=0 \text{,}
        \end{split}
    \end{equation}
    solving the Eq. (\ref{IC1}), we find that the initial conditions $\Theta_{i}(i=1,2)$ satisfy 
    \begin{equation}\label{ICLI}
     \Theta_{i}(x)= \Theta(x)=
     \begin{cases}
           \frac{\delta}{e_1 x+e_4},\hspace{1.3cm} e_2=0 \text{,}\\
           \delta (e_1 x+e_4)^{\frac{e_{3}}{e_{1}}}, \hspace{0.2cm} e_1 \neq 0 , e_2=0 \text{,}
    \end{cases} 
    \end{equation}
    where $\delta$ is a constant. 
Now, we can conclude from Eq. (\ref{ICLI}) that
\begin{enumerate}
    \item For $e_2=0$
\begin{equation}
   \mathbb{P}^1= (e_1 x +e_3 t+ e_4) \frac{\partial}{\partial_{x}}+ (e_1 t)\frac{\partial}{\partial_{t}}-(e_1 \rho) \frac{\partial}{\partial_{\rho}}+e_3 \frac{\partial}{\partial_{u}}\text{,}  
\end{equation}
the Eqs. (\ref{IC2}) satisfy the initial conditions
\begin{equation}
    \rho(x,0)=u(x,0)= \frac{\delta}{e_1 x+e_4}.
\end{equation}
\item For $e_1 \neq 0, e_2=0$
\begin{equation}
   \mathbb{P}^1= (e_1 x +e_3 t+ e_4) \frac{\partial}{\partial_{x}}+ (e_1 t)\frac{\partial}{\partial_{t}}-(e_1 \rho) \frac{\partial}{\partial_{\rho}}+e_3 \frac{\partial}{\partial_{u}}\text{,}  
\end{equation}
the Eqs. (\ref{IC2}) satisfy the initial conditions
\begin{equation}
    \rho(x,0)=u(x,0)=\delta (e_1 x+e_4)^{\frac{e_{3}}{e_{1}}}.
\end{equation}
\end{enumerate}
\subsection{Symmetry Group of the System}
\qquad We obtain the corresponding four one-parameter Lie group of transformations:
\begin{equation}
    \begin{aligned}
&\mathcal{G}_{1}:(x^{\star}, t^{\star}, \rho^{\star}, u^{\star})= (x e^\varepsilon, t e^\varepsilon, \rho e^{-\varepsilon}, u)\text{,}\\
&\mathcal{G}_{2}:(x^{\star}, t^{\star}, \rho^{\star}, u^{\star})= (x, t+\varepsilon, \rho, u)\text{,}\\
&\mathcal{G}_{3}:(x^{\star}, t^{\star}, \rho^{\star}, u^{\star})= (x+\varepsilon t, t, \rho, u+\varepsilon)\text{,}\\
&\mathcal{G}_{4}:(x^{\star}, t^{\star}, \rho^{\star}, u^{\star})= (x+\varepsilon, t, \rho, u)\text{.}
    \end{aligned}
\end{equation}

In view of the above one-parameter Lie group of transformations, we state the following:
\begin{thm}
Suppose that $\rho=m(x,t)$, $u=n(x,t)$ be a known solution of the given system \eqref{TM}. Then the group action of $\mathcal{G}_{1}, \mathcal{G}_{2}, \mathcal{G}_{3}, \mathcal{G}_{4}$ provides respectively the one-parameter family of solutions as
\begin{equation}
    \begin{aligned}
&\mathcal{G}_{1}:\hspace{0.3cm} \rho= e^{- \varepsilon}m(x e^{-\varepsilon},t e^{-\varepsilon}),\hspace{0.1cm} u=n(x e^{-\varepsilon},t e^{-\varepsilon}),\\
&\mathcal{G}_{2}:\hspace{0.3cm}\rho=m(x,t-\varepsilon),\hspace{0.1cm} u=n(x,t-\varepsilon),\\
&\mathcal{G}_{3}:\hspace{0.3cm}\rho=m(x-\varepsilon t,t),\hspace{0.1cm} u=\varepsilon+n(x-\varepsilon t,t),\\
&\mathcal{G}_{4}:\hspace{0.3cm}\rho=m(x-\varepsilon,t),\hspace{0.1cm} u=n(x-\varepsilon,t).
    \end{aligned}
\end{equation}
\end{thm}
\subsection{Calculation of Invariants}
\begin{lem}
The 2 arbitrary elements $\mathbb{A}=\sum_{i=1}^{4} \omega_{i} \mathscr{S}_{i}$ and $\mathbb{B}=\sum_{i=1}^{4} \gamma_{i} \mathscr{S}_{i}$ of $\mathcal{L}^4$   admit a general invariant function $I_{nv}(\omega_{1}, \omega_{2}, \omega_{3}, \omega_{4})=\mathbb{W}(\omega_{1}, \omega_{3})$, for $\omega, \gamma \in R$.
\end{lem}
\begin{proof}
The following system of PDEs can be derived by straightforward analysis from \cite{sharma2025analysis, bluman2010applications}:
\begin{equation}
\gamma_{1}: \omega_{2}\frac{\partial \pi}{\partial \omega_{2}}+\omega_{4} \frac{\partial \pi}{\partial \omega_{4}}=0\text{,} \quad \gamma_{2} : \omega_{1} \frac{\partial \pi}{\partial \omega_{2}}+\omega_{3} \frac{\partial \pi}{\partial \omega_{4}}=0\text{,}\quad \gamma_{3}:\omega_{2}\frac{\partial \pi}{\partial \omega_{4}}=0 \text{,} \quad \gamma_{4} : \omega_{1} \frac{\partial \pi}{\partial \omega_{4}}=0\text{.}
\end{equation}
which, on solving, produces $I_{nv}$.
\end{proof}
\subsection{Killing Form}
\begin{lem}
$K(\mathscr{S}, \mathscr{S})= 2 \omega_1^2$ is the Killing form for $L^{4}$.
\end{lem}
\begin{proof}
It is evident that
$
K(\mathscr{S},\mathscr{S})=Trace(\operatorname{ad(\mathscr{S})} \circ \operatorname{ad(\mathscr{S})})\text{,} 
$
where
$$
\begin{aligned}
& \operatorname{ad}(\mathscr{S})=\left[\begin{array}{cccc}
0 & 0 & 0 & 0  \\
-\omega_{2} & \omega_{1} & 0 & 0 \\
0 & 0 & 0 & 0 \\
-\omega_{4} & \omega_{3} & -\omega_{2} & \omega_{1}
\end{array}\right]\text{,}
\end{aligned} 
$$
which on simplification, generates $K(\mathscr{S}, \mathscr{S})=2 \omega_1^2$\text{.}
\end{proof}
\subsection{Construction of Adjoint Transformation Matrix}
\qquad The adjoint representation corresponding to $L^{4}$ is illustrated in Table \ref{AD} using the relation
\begin{equation}
\operatorname{Ad}(\exp (\varepsilon \mathscr{S}_{i}) \mathscr{S}_{j})  =\mathscr{S}_{i}-\varepsilon[\mathscr{S}_{i}, \mathscr{S}_{j}]+\frac{1}{2 !} \varepsilon^{2}[\mathscr{S}_{i},[\mathscr{S}_{i}, \mathscr{S}_{j}]]-\ldots\text{,}
\end{equation}
and the adjoint actions of the symmetries $\mathscr{S}_{i}$, where $i=1, \dots ,4$, operating on $\mathscr{S}$ provide the subsequent adjoint transformation matrices:
\begin{alignat*}{4}
K_{1}= &
\begin{pmatrix}
1 & 0 & 0 & 0\\
0 & e^{\varepsilon_{1}} & 0 & 0\\
0 & 0 & 1 & 0\\
0 & 0 & 0 & e^{\varepsilon_{1}}
\end{pmatrix}\text{,}  \quad&
K_{2}= & 
\begin{pmatrix}
1 & -\varepsilon_{2} & 0 & 0\\
0 & 1 & 0 & 0\\
0 & 0 & 1 & -\varepsilon_{2}\\
0 & 0 & 0 & 1
\end{pmatrix}\text{,} \quad&
K_{3}= &
\begin{pmatrix}
1 & 0 & 0 & 0 \\
0 & 1 & 0 & \varepsilon_{3}\\
0 & 0 & 1 & 0\\
0 & 0 & 0 & 1
\end{pmatrix}\text{,}  \quad&
     K_{4}= &
     \begin{pmatrix}
1 & 0 & 0 & -\varepsilon_{4} \\
0 & 1 & 0 & 0\\
0 & 0 & 1 & 0 \\
0 & 0 & 0 & 1
 \end{pmatrix}\text{.} 
\end{alignat*}
The adjoint action of $\mathscr{S}_{j}$ on $\mathscr{S}_{i}$ can be derived from the adjoint representation Table \ref{AD}. 

The algebraic equations are resolved to provide an optimal Lie algebra system.  The corresponding Lie subalgebras can be identified by employing the adjoint action on the set of these Lie subalgebras. Let
$$
\mathcal{T}=\omega_{1} \mathscr{S}_{1}+\omega_{2} \mathscr{S}_{2}+\omega_{3} \mathscr{S}_{3}+\omega_{4} \mathscr{S}_{4}\text {,}
$$
where $\omega_{1}, \omega_{2}, \omega_{3}$, and $\omega_{4}$ are real constants. In this context, $\mathcal{T}$ can be expressed as a column vector with elements $\omega_{1}, \omega_{2}, \omega_{3}, \omega_{4}$.  
\begin{equation}
\mathcal{K}\left(\varepsilon_{1}, \varepsilon_{2}, \varepsilon_{3}, \varepsilon_{4}\right) = K_{4} K_{3} K_{2} K_{1} \text{.}
\end{equation}
The above equation gives
$$
\mathcal{K}=\left[\begin{array}{cccc}
1 & -\varepsilon_{2} e^{\varepsilon_{1}} & 0 & -\varepsilon_{4} e^{\varepsilon_{1}}\\
0 & e^{\varepsilon_{1}} & 0 & \varepsilon_{3} e^{\varepsilon_{1}}\\
0 & 0 & 1 & -\varepsilon_{2} e^{\varepsilon_{1}} \\
0 & 0 & 0 & e^{\varepsilon_{1}}  
\end{array}\right]\text{.} 
$$
\renewcommand{\arraystretch}{2.0}
\begin{table}[H]
    \centering
    \begin{tabular}{c|cccc}\hline 
    \multicolumn{5}{c}{Cofficients of $\mathscr{S}_{i}$} \\\hline \hline
        $\operatorname{Ad}_{\exp{(\varepsilon \mathscr{S}_i)}}\mathcal{T}$ & $\mathscr{S}_1$ &$ \mathscr{S}_2$ & $\mathscr{S}_3$ & $\mathscr{S}_4$\\ \hline \hline
        $\operatorname{Ad}_{\exp{(\varepsilon \mathscr{S}_1)}}\mathcal{T})$ & $\omega_1$ & $e^{\varepsilon}\omega_2$ & $\omega_3$ & $e^{\varepsilon} \omega_4$ \\ 
        $\operatorname{Ad}_{\exp{(\varepsilon \mathscr{S}_2)}}\mathcal{T})$ & $\omega_1$ & $\omega_2-\varepsilon \omega_1$ & $\omega_3$ & $\omega_4-\varepsilon \omega_{3}$ \\ 
        $\operatorname{Ad}_{\exp{(\varepsilon \mathscr{S}_3)}}\mathcal{T})$ & $\omega_1$ & $\omega_2$ & $\omega_3$ & $\omega_4 +\varepsilon \omega_2$ \\ 
        $\operatorname{Ad}_{\exp{(\varepsilon \mathscr{S}_4)}}\mathcal{T})$ & $\omega_1$ & $\omega_2$ & $\omega_3$ & $\omega_4-\varepsilon \omega_1$ \\ \hline \hline
    \end{tabular}
    \caption{Table for Invariant Function.}
    \label{t2}
\end{table}
\begin{lem}
   The subalgebra of $\mathcal{L}^4$ generated by $\mathscr{S}_1$ and $\mathscr{S}_3$, denoted as $\mathcal{M}=\omega_1$ and $\mathcal{N}=\omega_3$, respectively, is $\operatorname{Ad}$-invariant.
\end{lem}
\begin{proof}
   The data in Table \ref{t2} demonstrates the evidence.
\end{proof}
\begin{lem}
Let $\mathcal{P}$ be the function with respect to the basis $\{\mathscr{S}_i \mid 1 \leq i \leq 4\}$, defined as $\mathcal{P}:\mathcal{L}^4 \rightarrow \mathbb{R}$ as follows:
\[
\mathcal{P} = 
\begin{cases}
1, & \text{if } \omega_1^2 + \omega_2^2 + \omega_3^2 \neq 0, \\
0, & \text{otherwise}.
\end{cases}
\]
Then under the action of the Lie algebra $\mathcal{L}^4$, $\mathcal{P}$ is invariant.
\end{lem}

\begin{proof}
From Table \ref{t2}, it is evident that the coefficients associated with the symmetries $\mathscr{S}_1$, $\mathscr{S}_2$, and $\mathscr{S}_3$ remain unchanged under the adjoint transformations $\operatorname{Ad}_{\exp(\varepsilon \mathscr{S}_j)}$ for $j = 3, 4$. Therefore, it is sufficient to verify that the function $\mathcal{P}$ is invariant under adjoint actions $\operatorname{Ad}_{\exp(\varepsilon \mathscr{S}_1)}$ and $\operatorname{Ad}_{\exp(\varepsilon \mathscr{S}_2)}$.

Let $\omega_i^\star$, for $i = 1, 2, 3, 4$, represent the transformed coefficients following the execution of these adjoint operations.  The following relationships are established:
\[
\begin{aligned}
& (\omega_1^\star)^2 + (\omega_2^\star)^2 + (\omega_3^\star)^2 = \omega_1^2 + (e^{\varepsilon} \omega_2)^2 + \omega_3^2, \\
& (\omega_1^\star)^2 + (\omega_2^\star)^2 + (\omega_3^\star)^2 = \omega_1^2 + (\omega_2 - \varepsilon \omega_1)^2 + \omega_3^2.
\end{aligned}
\]

In both cases, it is clear that the condition $\omega_1^2 + \omega_2^2 + \omega_3^2 = 0$ holds if and only if $(\omega_1^\star)^2 + (\omega_2^\star)^2 + (\omega_3^\star)^2 = 0$. Thus, under the adjoint action of the Lie algebra $\mathcal{L}^4$, the function $\mathcal{P}$ remains invariant.
\end{proof}
\begin{lem}\label{LEM1}
Let the function $\mathcal{Q}: \mathcal{L}^4 \rightarrow \mathbb{R}$ be defined with respect to the basis $\{\mathscr{S}_i \mid 1 \leq i \leq 4\}$. Then $\mathcal{Q}$ is invariant under the adjoint action within the subalgebra generated by $\{\mathscr{S}_2, \mathscr{S}_3, \mathscr{S}_4\}$, i.e., $
\mathcal{Q}(\mathcal{T}) = \mathcal{Q}(\operatorname{Ad}_g(\mathcal{T})),
$
where
\[
\mathcal{Q} = 
\begin{cases}
\operatorname{sgn}(\omega_2), & \text{if } \omega_1 = 0, \\
0, & \text{otherwise}.
\end{cases}
\]
\end{lem}

\begin{proof}
As shown in Table \ref{t2}, the coefficient $\omega_2$ corresponding to the generator $\mathscr{S}_2$ remains unchanged under the adjoint transformations $\operatorname{Ad}_{\exp(\varepsilon \mathscr{S}_i)}$ for $i = 3, 4$. Therefore, it is sufficient to verify the invariance of $\mathcal{Q}$ under the adjoint actions for $i = 1,2$.

In the case $i = 2$, the condition $\mathcal{Q}(\mathcal{T}) = \mathcal{Q}(\operatorname{Ad}_g(\mathcal{T}))$ is directly satisfied. For $i = 1$, the transformation maps $\omega_2$ to $e^{\varepsilon} \omega_2$. Since this transformation preserves the sign of $\omega_2$, the function $\mathcal{Q}$ remains invariant under the action. This completes the proof.
\end{proof}
\begin{lem}
Let the function $\mathcal{R}: \mathcal{L}^4 \rightarrow \mathbb{R}$ be defined with respect to the basis $\{\mathscr{S}_i \mid 1 \leq i \leq 4\}$. Then $\mathcal{R}$ is invariant under the adjoint action within the subalgebra generated by $\{\mathscr{S}_2, \mathscr{S}_3, \mathscr{S}_4\}$, i.e., $
\mathcal{R}(\mathcal{T}) = \mathcal{R}(\operatorname{Ad}_g(\mathcal{T})),$
where
\[
\mathcal{R} = 
\begin{cases}
\operatorname{sgn}(\omega_4), & \text{if } \omega_1=\omega_2=\omega_3 = 0, \\
0, & \text{otherwise}.
\end{cases}
\]
\end{lem}

\begin{proof}
The proof follows by analogy with Lemma (\ref{LEM1}).
\end{proof}
\subsection{Construction of Optimal System}
\quad \quad To construct an optimal system of the system \eqref{TM} utilizing the approach mentioned in \cite{hu2015direct}, we consider $\mathbb{A}=\sum_{i=1}^{4} \omega_{i} \mathscr{S}_{i}$ and $\mathbb{B}=\sum_{i=1}^{4} q_{i} \mathscr{S}_{i}$ as two elements of $\mathcal{L}^4$. The model's adjoint transformation equation is given by
\begin{equation}\label{urv}
\left(q_{1}, q_{2}, q_{3}, q_{4}\right)=\left(\omega_{1}, \omega_{2}, \omega_{3}, \omega_{4}\right) \mathcal{K} \text{.}
\end{equation}

Further, $\mathcal{K}\left(\varepsilon_{1}, \varepsilon_{2},\varepsilon_{3} ,\varepsilon_{4}\right)$ gives a transformed $\mathcal{T}$ as follows:
\begin{equation}\label{424}
\begin{split}
\mathcal{K}\left(\varepsilon_{1}, \varepsilon_{2}, \varepsilon_{3} ,\varepsilon_{4},\right) \cdot \mathcal{T}=\hspace{0.1cm} & \omega_{1} \mathscr{S}_{1}+\left(-\omega_{1} \varepsilon_{2} +\omega_{2} \right) e^{\varepsilon_{1}} \mathscr{S}_{2}+\omega_{3} \mathscr{S}_{3}+ \left(-\omega_{1} \varepsilon_{4}+\omega_{2} \varepsilon_{3}- \varepsilon_{2}\omega_{3}+\omega_{4}\right) e^{\varepsilon_{1}} \mathscr{S}_{4} \text{.} 
    \end{split}
    \end{equation}
We evaluate the subsequent four cases following the invariants' degree:
$$
\begin{aligned}
\mathcal{R}_{1}=\left\{\omega_{1}=0,
\omega_{3}=1\right\} \text{,} \hspace{0.2cm} \mathcal{R}_{2}=\left\{\omega_{1}=1,\omega_{3}=0\right\}\text{,} \hspace{0.2cm} \mathcal{R}_{3}=\left\{\omega_{1}=l_{1}, \omega_{3}=l_{2}\right\} \text{,}\hspace{0.2cm} \mathcal{R}_{4}=\left\{\omega_{1}=0, \omega_{3}=0\right\} \text {, }\hspace{0.1cm} \text{with} \hspace{0.1cm} l_{1}, l_{2} \neq 0\text{.}
\end{aligned}
$$

Let us consider the case $\mathcal{R}_{1}: \omega_{1}=0\text{,} \hspace{0.1cm} \omega_{3}=1\text{.}$ Choosing a representative element
$
\hat{\mathcal{T}}=\omega_{2} \mathscr{S}_{2}+ \mathscr{S}_{3} +\omega_{4} \mathscr{S}_{4} \text {,} 
$
and substitute $q_{i}=0\text{,} \hspace{0.2cm} \text{for} \hspace{0.2cm} i=1,2,4 \hspace{0.2cm} \text{and}\hspace{0.2cm} q_{3}=1$ in Eqs. \eqref{urv} and \eqref{424} we get the solution as $\varepsilon_{2}=\omega_{4}$.
Thus, the action of adjoint maps $\operatorname{Ad}_{\exp \left(\varepsilon_{2} \mathscr{S}_{2}\right)}$ will eliminate the coefficients of $X_{2}$, from $\hat{\mathcal{T}}$ and 
we get $\hat{\mathcal{T}}=\mathscr{S}_{3}+\omega_{4} \mathscr{S}_{4}
$. Finally, the action of \(\operatorname{Ad}_{\exp \left(\varepsilon_{1} \mathscr{S}_{1}\right)}\) on \(\hat{\mathcal{T}}\) reduces the general element \(\mathcal{T}\) to the form \(\hat{\mathcal{T}} = \mathscr{S}_{3} + b \mathscr{S}_{4}\), where \(b \in \{-1, 0, 1\}\), thereby preventing any further reduction.
Similarly, the one-dimensional optimal subalgebras for the remaining cases can be determined using the same procedure, and the results are summarized in the following theorem:

\begin{thm}
The one-dimensional optimal subalgebras for the Lie algebra $\mathbb{L}^{4}$ of the system (\ref{TM}) are as follows:
$$
\begin{aligned}
&\mathcal{T}_{1, b} = \mathscr{S}_{3} + b \mathscr{S}_{4}\text{,}\\ &\mathcal{T}_{2, b} = \mathscr{S}_{1} + b \mathscr{S}_{4}\text{,}\\
&\mathcal{T}_{3, l_1, l_2, b} = l_{1} \mathscr{S}_{1} + l_{2} \mathscr{S}_{3} + b \mathscr{S}_{4}\text{,} \\
&\mathcal{T}_{4, b} = \mathscr{S}_{2} + b \mathscr{S}_{4}\text{,}
\end{aligned}
$$
where $b$ is a parameter with $b\in \{-1, 0, 1\}$. In addition, $l_1, l_2 \neq 0$ are arbitrary constants.   
\end{thm}
\begin{proof}
It is sufficient to prove the theorem by demonstrating that the $\mathscr{T}_i^{'} s$ are distinct.  To achieve this, we will construct a table that illustrates the invariants in Table \ref{t4}. An analysis of the contents of Table \ref{t4} reveals that the subalgebras are different.
    \renewcommand{\arraystretch}{1.8}
\begin{table}[H]
    \centering
    \begin{tabular}{c|cccccc}
    \hline \hline
      & $K$ & $\mathcal{M}$ & $\mathcal{N}$ & $\mathcal{P}$ & $\mathcal{Q}$ & $\mathcal{R}$ \\\hline \hline
 $\mathcal{T}_{1,b}(\mathscr{S}_3+b \mathscr{S}_4)$  & 0 & 0 & 1 & 1 & 0 & 0\\
$\mathcal{T}_{2,b}(\mathscr{S}_1+b \mathscr{S}_4)$  & 2 & 1 & 0 & 1 & 0 & 0\\
$\mathcal{T}_{3,1,1,b}(\mathscr{S}_1+\mathscr{S}_3+b \mathscr{S}_4)$  & 2 & 1 & 1 & 1 & 0 & 0\\
 $\mathcal{T}_{4,b}(\mathscr{S}_2+b \mathscr{S}_4)$  & 0 & 0 & 0 & 1 & 1 & 0\\
        \hline \hline
    \end{tabular}
    \caption{Table for Invariant Function Evolution.}
    \label{t4}
\end{table}
\end{proof}
\section{Similarity Reduction and Group Invariant Solutions}
\subsection{Invariant Solutions via Optimal System}
\subsubsection{\texorpdfstring{$\mathcal{T}_{3} = l_{1} \mathscr{S}_{1} + l_{2} \mathscr{S}_{3} + b \mathscr{S}_{4}$}{T3 = l1 S1 + l2 S3 + b S4}}
The characteristic equation can be written as:
\begin{equation}
    \frac{d x}{l_1 x+ l_2 t+b}=\frac{d t}{l_1 t}=\frac{d \rho}{-l_1 \rho}=\frac{d u}{l_2}\text{,}
\end{equation}
which, on solving, gives the similarity variables as follows:
\begin{equation}\label{C3}
  \tau=\frac{-l_2 t \ln(t)+ l_1 x+b}{l_1 t},\hspace{0.2cm} \rho=\frac{\mathcal{Y}(\tau)}{t},\hspace{0.2cm} u=\frac{l_2 \ln(t)}{l_1}+\mathcal{Z}(\tau) \text{.}
\end{equation}
Substituting the Eq. \eqref{C3} in Eq. \eqref{TM}, we get the following ODEs:
\begin{equation}
    \begin{split}
& \mathcal{Y}(\mathcal{Z}^{'} -1)+\mathcal{Y}^{'}( \mathcal{Z}-\frac{l_2}{l_1}-\tau)=0\text{,}\\
& \frac{l_2}{l_1}+\mathcal{Z}^{'}(\mathcal{Z}-\frac{l_2}{l_1}-\tau)+\frac{1}{\mathcal{Y}}(A \mathcal{Y}^{'}-D \mathcal{Z}^{''})=0\text{.}
    \end{split}
\end{equation}
For $l_1=1, l_2=1$, and $D=0$, we get the solutions as follows:
\begin{equation}\label{o3}
    \mathcal{Z}(\tau)=\tau+1, \quad \mathcal{Y}(\tau)= p_1 e^{\frac{-\tau}{A}},
\end{equation}
where $p_1$ is an arbitrary constant.
From Eqs. \eqref{o3} and \eqref{C3}, we get the solutions as follows:
\begin{equation}\label{E533}
\rho=\frac{p_1}{t}e^{\frac{t \ln(t)-x-b}{t A}}, \quad u=\frac{x+b}{t}+1\text{.}    
\end{equation}
Similarly, the solutions for the remaining cases can be derived using the same approach, as presented in Table \ref{OST}.
\renewcommand{\arraystretch}{2.5}
\begin{table}[H]
    \centering
    \begin{tabular}{>{\centering\arraybackslash}m{2.5cm} >{\centering\arraybackslash}m{3.5cm} >{\centering\arraybackslash}m{6cm} >{\centering\arraybackslash}m{4.5cm}}    
    \hline \hline
    \textbf{Subalgebras} & \textbf{Similarity Variables} & \textbf{Reduced Equations} & \textbf{Invariant Solutions}  \\
    \hline \hline
     $\mathcal{T}_{1} = \mathscr{S}_{3} + b \mathscr{S}_{4}$ & 
    $ \tau=t,\hspace{0.1cm} \rho=\mathcal{Y}(\tau),$ \vspace{1em} \newline$u=\frac{x}{\tau+b}+\mathcal{Z}(\tau)$ & 
    $\mathcal{Z}_{\tau}+\frac{\mathcal{Z}}{\tau+b}=0$, \vspace{1em}\newline $\mathcal{Y}_{\tau}+\frac{\mathcal{Z}}{\tau+b}=0$ & $\rho=\frac{p_2}{t+b},$ \vspace{1em} \newline$ u=\frac{x+p_1}{t+b}$\\\\
    $\mathcal{T}_{2} = \mathscr{S}_{1} + b \mathscr{S}_{4}$, & $ \tau=\frac{x+b}{t}$ \vspace{1em} \newline $\rho=\frac{\mathcal{Y}(\tau)}{t}, \hspace{0.1cm} u=\mathcal{Z}(\tau)$ & $\mathcal{Y}^{'} \mathcal{Z}+\mathcal{Y} \mathcal{Z}^{'}-\mathcal{Y}^{'}\tau-\mathcal{Y}=0$, \vspace{1em} \newline $\mathcal{Y} \mathcal{Z}^{'} \tau+\mathcal{Y} \mathcal{Z}\mathcal{Z}^{'}+A \mathcal{Y}^{'}-D \mathcal{Z}^{''}=0$ & For $D=0$,\vspace{1em} \newline $\rho=\frac{2 p_1}{-(x+b)+\sqrt{(x+b)^2-4 A t^2}}$, \vspace{1em} \newline$u=\frac{x+b+\sqrt{(x+b)^2-4 A t^2}}{2 t}$ \\ \\
  $\mathcal{T}_{4} =\mathscr{S}_{2} + b\mathscr{S}_{4}$ & $ \tau=x-\frac{b}{t}$,\vspace{1em} \newline $\rho=\mathcal{Y}(\tau),\hspace{0.1cm} u=\mathcal{Z}(\tau)$ & $-D\mathcal{Z}^{''}+\mathcal{Y} \mathcal{Z}^{'}(\mathcal{Z}-b)+A \mathcal{Y}^{'}=0$, \vspace{1em} \newline$\mathcal{Y} \mathcal{Z}^{'} +\mathcal{Y}^{'} \mathcal{Z}-b\mathcal{Y}^{'}=0$ & For $D=0$, $\rho=\frac{p_1}{\sqrt{A}}$,\vspace{1em} \newline $u=b+\sqrt{A}$\\
    \hline \hline
    \end{tabular}
    \caption{Similarity Reductions and Invariant Solutions.}
    \label{OST}
\end{table}
\subsection{Invariant Solutions via Parameter Analysis}
\begin{equation}
    \frac{dx}{\mathscr{F}_{x}}= \frac{dt}{\mathscr{F}_{t}}= \frac{d\rho}{\mathscr{F}_{\rho}}= \frac{d u}{\mathscr{F}_{u}}\text{.}
\end{equation}
\subsubsection{\texorpdfstring{$e_1=0$}{e1=0}}
The similarity variables for this case are derived as follows:
\begin{equation}\label{C5}
  \tau=\frac{e_3 t^2}{2}+e_4 t-e_2 x,\hspace{0.2cm} \rho=\mathcal{Y}(\tau),\hspace{0.2cm} u=\frac{e_3 t}{e_2}+\mathcal{Z}(\tau)\text{.}
\end{equation}
Substituting the Eq. \eqref{C5} in Eq. \eqref{TM}, we get the following ODEs:
\begin{equation}
    \begin{split}
& \frac{e_3}{e_2}+e_4\mathcal{Z}^{'}-e_2 \mathcal{Z} \mathcal{Z}^{'}-\frac{e_2 A \mathcal{Y}^{'}}{\mathcal{Y}}-\frac{e_2^2 D \mathcal{Z}^{''}}{\mathcal{Y}}\text{,}\\
& e_2 \mathcal{Y} \mathcal{Z}^{'} + e_2 \mathcal{Y}^{'} \mathcal{Z}-e_4\mathcal{Y}^{'}=0\text{.}
    \end{split}
\end{equation}
For $D=A=0$, we get the solutions of the above equations as follows:
\begin{equation}\label{o5}
    \mathcal{Z}(\tau)=\frac{e_4-\sqrt{e_4^2+2 e_3 p_2 +2 e_3 \tau}}{e_2}, \quad \mathcal{Y}(\tau)=\frac{p_1}{\sqrt{e_4^2+2 e_3 p_2 +2  e_3 \tau}},
\end{equation}
where $p_1$ and $p_2$ are arbitrary constants.
From Eqs. \eqref{o5} and \eqref{C5}, we get solutions as follows:
\begin{equation}\label{522}
\rho=\frac{p_1}{\sqrt{2 e_3 p_2+(e_3 t+e_4)^2 -2 e_2 e_3 x}}, \quad u=\frac{e_3 t}{e_2}+\frac{e_4-\sqrt{2 e_3 p_2+(e_3 t+e_4)^2 -2 e_2 e_3 x}}{e_2}\text{.}    
\end{equation}
\subsubsection{\texorpdfstring{$e_3=0$}{e3=0}}
The similarity variables for this case are derived as follows:
\begin{equation}\label{C6}
  \tau=\frac{e_1 x+e_4}{e_1 t+
  e_2},\hspace{0.2cm} \rho=\frac{\mathcal{Y}(\tau)}{e_1 t+e_2},\hspace{0.2cm} u=\mathcal{Z}(\tau)\text{.}
\end{equation}
Substituting the Eq. \eqref{C6} in Eq. \eqref{TM}, we get the following ODEs:
\begin{equation}
    \begin{split}
& -\tau \mathcal{Z}^{'}+ \mathcal{Z} \mathcal{Z}^{'}+\frac{A \mathcal{Y}^{'}}{\mathcal{Y}}-\frac{e_1 D \mathcal{Z}^{''}}{\mathcal{Y}}\text{,}\\
& \mathcal{Y} \mathcal{Z}^{'} + \mathcal{Y}^{'} \mathcal{Z}-\tau \mathcal{Y}^{'}-\mathcal{Y}=0\text{.}
    \end{split}
\end{equation}
For $D=A=0$, we get the solutions of the above equations as follows:
\begin{equation}\label{o6}
    \mathcal{Z}(\tau)=\frac{\tau+\sqrt{\tau^2-4 A}}{2}, \quad \mathcal{Y}(\tau)=\frac{2 p_1}{-\tau+\sqrt{\tau^2-4A}},
\end{equation}
where $p_1$ is an arbitrary constant.
From Eqs. \eqref{o6} and \eqref{C6}, we get solutions as follows:
\begin{equation}
\rho=\frac{2 p_1}{-(e_1 x+e_4)+\sqrt{(e_1 x+e_4)^2-4 A (e_1 t+e_2)^2}}, \quad u=\frac{e_1 x+e_4+\sqrt{(e_1x+e_4)^2-4 A (e_1 t+e_2)^2}}{2 (e_1 t +e_2)}\text{.}    
\end{equation}

\section{Conservation-Based Exact Solutions}
In this section, we derive the exact solutions of the system \eqref{TM} using the conservation laws of mass and momentum \cite{avdonina2013exact,shagolshem2024exact}. The conservation is of the form
$$D_{t}(T^i)+D_{x}(C^i)=0,$$ for $i=1,2$.
We can rewrite the system \eqref{TM} for $D=0$ as follows:
\begin{equation}\label{SCL1}
\begin{split}
   &\mathbb{R}_{1}: \rho u_{x}+\rho_{x} u+ \rho_{t} \equiv D_{t}(\rho)+D_{x}(\rho u)=0,\\
   & \mathbb{R}_{2}: u_{t}+u u_{x}+\frac{A \rho_{x}}{\rho}\equiv D_{t}(\rho u)+D_{x}(\rho u^{2}+A \rho)=0\text{.}
\end{split}
\end{equation}
The corresponding conserved vector for the $\mathbb{R}_1$ has the components
$T^1=\rho$ and $C^1=\rho u$.
The differential constraints \eqref{SCL1} are written $D_{t}(\rho)=D_{x}(\rho u)=0$ and yield the following equations:
$$
\rho=M(x),\hspace{0.1cm} \rho u=N(t).
$$
Hence, we look for solutions of the form
\begin{equation}\label{SCL945}
    \rho=M(x),\hspace{0.1cm} u=\frac{N(t)}{M(x)}.
\end{equation}
Substituting the solutions from Eq. \eqref{SCL945} into the $\mathbb{R}_{2}$ of Eq. \eqref{SCL1}, we get the ODE as follows:
\begin{equation}\label{SCL946}
    M^{2} N^{'}-N^{2} M^{'}+A M^{2} M^{'}=0.
\end{equation}
Solution of the Eq. \eqref{SCL946} is as follows:
\begin{equation}
    N(t) = -\tanh\left(\sqrt{A} \frac{dM(x)}{dx} \frac{(c_1 + t)}{M(x)}\right) M(x) \sqrt{A}.
\end{equation}
So, from the above equation, we get 
\begin{equation}\label{CLS}
    u(x,t)= -\tanh\left(\sqrt{A}  \frac{\frac{dM(x)}{dx}(c_1 + t)}{M(x)}\right) \sqrt{A}, \hspace{0.3cm} \rho(x,t)=M(x).
\end{equation}
The above solutions (\ref{CLS}) are exact soliton-type solutions to the given hydrodynamic model (\ref{TM}). It's a kink-type soliton, but still localized and smooth. The solutions corresponding to different values of the function $M(x)$ are presented in Table \ref{CBES} as follows:
\renewcommand{\arraystretch}{2.0}
\begin{table}[H]
    \centering
    \begin{tabular}{c|c|c|c}
    \hline \hline
       \text{Sr. No.}& $M(x)$& $\rho(x,t)$ &$u(x,t)$ \\ \hline \hline
      1. & $\sin{x}$  & $\sin{x}$  & $-\sqrt{A} \tanh\left(\cot{x} (c_1 + t) \sqrt{A}  \right)$\\
        2. & $\sec{x}$ & $\sec{x}$ & $-\sqrt{A} \tanh\left( \tan{x} (c_1 + t) \sqrt{A} \right)$\\
       3. & $\cos{x}$ & $\cos{x}$ & $\sqrt{A} \tanh\left(\tan{x}(c_1 + t) \sqrt{A}  \right)$ \\
       4. & $e^{-x^2}$ & $e^{-x^2}$& $\sqrt{A} \tanh\left( 2x(c_1 + t) \sqrt{A}  \right)$ \\
       \hline\hline
    \end{tabular}
    \caption{Exact Solutions via Conserved Vector.}
    \label{CBES}
\end{table}
\section{Conserved Vectors}
\qquad Ibragimov \cite{ibragimov2006integrating, ibragimov2007new,ibragimov2011nonlinear} initially proposed the concept of nonlinear self-adjointness for the system of partial differential equations (PDEs).  This method is extremely efficient for systematically deriving conservation rules, regardless of the existence of a variational principle in the given system of partial differential equations (PDEs). This section demonstrates that the Eq. \eqref{TM} exhibits nonlinear self-adjointness.
\subsection{Nonlinear Self-adjointness }
\begin{thm}
The given model in Eq. \eqref{TM} has nonlinear self-adjointness.
\end{thm}
\begin{proof}
The adjoint equation associated with \eqref{TM} can be formulated following the methodology explained in \cite{ibragimov2011nonlinear}
\begin{equation}
\begin{split}
&S_1=\frac{\delta L}{\delta u}= \frac{1}{\rho^{3}}[D g \rho \rho_{xx}-D g_{xx}\rho^{2}-2 D g \rho_{x}^{2}+2 D g_{x} \rho \rho_{x}-\rho^{3}(h_{x} \rho + g_{x} u+g_{t})]\text{,}\\
&S_2=\frac{\delta L}{\delta \rho}=\frac{1}{\rho^2}[-h_{x} u \rho^2-A g_{x} \rho+D g u_{xx}-h_{t} \rho^{2}]\text{,}\\
\end{split}
\end{equation}
where $L$ denotes the formal Lagrangian expressed as $L=h(x,t) \mathbb{R}_{1}+g(x,t) \mathbb{R}_{2}$, where $g(x,t)$ and $h(x,t)$ are the dependent variables. The Euler-Lagrange operators $\frac{\delta }{\delta u}$ and $\frac{\delta }{\delta \rho}$ pertaining to $u$ and $\rho$ are as follows:
\begin{equation}
\begin{split}
&\frac{\delta}{\delta u} = \frac{\partial}{\partial u} - D_t \frac{\partial}{\partial u_t} - D_x \frac{\partial}{\partial u_x} + D_x^2 \frac{\partial}{\partial u_{xx}} + D_x D_t \frac{\partial}{\partial u_{xt}} + D_t^2 \frac{\partial}{\partial u_{tt}} + \cdots \text{,}\\
&\frac{\delta}{\delta \rho} = \frac{\partial}{\partial \rho} - D_t \frac{\partial}{\partial \rho_t} - D_x \frac{\partial}{\partial \rho_x} + D_x^2 \frac{\partial}{\partial \rho_{xx}} + D_x D_t \frac{\partial}{\partial \rho_{xt}} + D_t^2 \frac{\partial}{\partial \rho_{tt}} + \cdots \text{,}
\end{split}
\end{equation}
where $D_{x}$ and $D_{t}$ denote the total differential operators with regard to $x$ and $t$, respectively, as defined below:
\begin{equation}
   \begin{split}
      & D_x = \frac{\partial}{\partial x} + w^{i}_x \frac{\partial}{\partial w^{i}} + w^{i}_{xx} \frac{\partial}{\partial w^{i}_x} + w^{i}_{xy} \frac{\partial}{\partial w^{i}_y} + w^{i}_{xt} \frac{\partial}{\partial w^{i}_t} + \cdots ,\\
& D_t = \frac{\partial}{\partial t} + w^{i}_t \frac{\partial}{\partial w^{i}} + w^{i}_{tt} \frac{\partial}{\partial w^{i}_t} + w^{i}_{xt} \frac{\partial}{\partial w^{i}_x} + w^{i}_{yt} \frac{\partial}{\partial w^{i}_y} + \cdots \text{,}
   \end{split} 
\end{equation}
where $w^{1}=u$ and $w^{2}=\rho$.
The given Eq. \eqref{TM} is nonlinearly self-adjoint \cite{ibragimov2011nonlinear} if it satisfies the following condition: 
\begin{equation}\label{NS1}
\begin{split}
&S_{1} \mid_{(h=H(x,t,u,\rho)), g=G(x,t,u,\rho)} = l_{1} \mathbb{R}_{1}+l_2 \mathbb{R}_{2} \text{,}\\
&S_{2} \mid_{(h=H(x,t,u,\rho)), g=G(x,t,u,\rho)} = l_{3} \mathbb{R}_{1}+l_4 \mathbb{R}_{2} \text{,}
\end{split}
\end{equation}
where $H(x,t,u,\rho)$ and $G(x,t,u,\rho)$ are nonzero functions, $l_1, l_2, l_3$ and $l_4$ are unknown to be determined by equating coefficient of $u_{t}$ and $\rho_{t}$ to zero. We obtain
\begin{equation}
    l_{1}=-g_{\rho},\quad l_2=-g_{u}, \quad l_3=-h_{\rho}, \quad l_4=-h_{u}\text{.}
\end{equation}
Substituting the value of $l_{1}, l_2, l_3$ and $l_4$ into \eqref{NS1}, one can obtain
\begin{equation}\label{NS2}
h=c_1 u-\frac{c_1A}{\rho}+c_3 \text{,} \hspace{0.2cm} \text{and} \hspace{0.2cm} g=(\rho+u)c_{1}+c_{2}\text{,}
\end{equation}
where $c_{i}, i=1,2,3$ are the arbitrary constants. Hence, the given Eq. \eqref{TM} is nonlinearly self-adjoint under \eqref{NS1}.
\end{proof}
\subsection{Conservation Laws and their Applications}
\quad \quad Conservation laws are essential in modeling systems where certain physical quantities remain constant over time. In PDEs, they relate to integrability and help analyze the existence, uniqueness, and stability of solutions, while also guiding the development of accurate numerical methods. Moreover, conservation laws can lead to the introduction of nonlocal variables, which in turn generate a related nonlocal PDE system useful for exploring nonlocal symmetries \cite{sil2020nonlocal, liu2020existence, shagolshem2023classification,friedrich2022conservation} and hidden structures within the original system.
\subsection{Deriving Conservation Laws via Nonlinear Self-Adjointness}
\quad \quad Ibragimov's \cite{ibragimov2011nonlinear} formalism provides an explicit framework for generating conserved vectors corresponding to admitted symmetries of a differential system, yielding a specific subclass of conservation laws. In contrast, the multipliers method developed by Anco and Bluman \cite{bluman2010local} offers a more general and systematic procedure, capable of recovering the complete set of local conservation laws. Nevertheless, due to the apparent inapplicability of the multipliers approach to the present governing equations, we adopt Ibragimov’s nonlinear self-adjointness method to systematically construct conservation laws associated with each Lie point symmetry.\\
Let $\Upsilon=(\Upsilon^{x}, \Upsilon^{t})$ be the conserved vector corresponding to $x$ and $t$, satisfying the conservation form
\begin{equation}
    D_{x} \Upsilon^{x}+D_{t} \Upsilon^{t}=0\text{,}
    \end{equation}
and generated by the infinite symmetry 
\begin{equation}
    \mathbb{P}=\mathscr{F}_{x}\frac{\partial}{\partial x}+\mathscr{F}_{t}\frac{\partial}{\partial t}+\mathscr{F}_{\rho}\frac{\partial}{\partial \rho}+\mathscr{F}_{u} \frac{\partial}{\partial u}\text{,}
\end{equation}
where the conserved vectors $\eta^{x},\eta^{y}$ and $\eta^{t}$ can be obtained from ideas developed in \cite{ibragimov2011nonlinear}
\begin{equation}\label{e766}
\begin{split}
& \Upsilon^{x}=\mathscr{F}_{x} L+\omega^1 \frac{\partial L}{\partial \rho_{x}}+\omega^2 \frac{\partial L}{\partial u_{x}}-\omega^2 D_{x} (\frac{\partial L}{\partial u_{xx}})+\frac{\partial L}{\partial u_{xx}}D_{x}(\omega^2)\text{,}\\
& \Upsilon^{t}=\mathscr{F}_{t} L+\omega^{1} \frac{\partial L}{\partial \rho_{t}}+\omega^{2} \frac{\partial L}{\partial u_{t}}\text{,}
\end{split}
\end{equation}
with $\omega^1=\mathscr{F}_{\rho}-\mathscr{F}_{x}\rho_{x}-\mathscr{F}_{t}\rho_{t}$ and $\omega^2=\mathscr{F}_{u}-\mathscr{F}_{x} u_{x}-\mathscr{F}_{t} u_{t}$. Furthermore, we determine the conserved vectors for the infinitesimal generators $\mathscr{S}_{1}, \mathscr{S}_{2}, \mathscr{S}_{3}$, and $\mathscr{S}_{4}$ in the following cases:\\

 \renewcommand{\arraystretch}{1.5}
 \begin{table}[H]
		\resizebox{\columnwidth}{!}{%
			\begin{tabular}{c|c}
				\hline \hline
                \textbf{Vector} & \textbf{Conserved Vector}  \\ \hline \hline
\multirow{3}{*} & $\Upsilon^{x}=\frac{D (c_1 \rho +c_1 u+c_2)}{\rho}(u_{x}+x u_{xx}+t u_{tx})+(x u_{x}+t u_{t})(-\frac{c_1 D u_{x}}{\rho}+\frac{D \rho_{x}}{\rho^2}(c_1 u+c_2)+(c_1 u -\frac{c_1 A}{\rho}+c_3)\rho+$\\
$\mathscr{S}_{1}=x \partial_{x}+t \partial_{t}-\rho \partial_{\rho}$ & $(c_1 u+c_1 \rho +c_2)u)-(\rho+x\rho_{x}+t\rho_{t})((c_1 u+c_3)u+\frac{(c_1 \rho +c_2) A}{\rho})$\\ 
& $\Upsilon^{t}=-(c_1(\rho+u)+c_2)(x u_{x}+t u_{t})-(c_1 u-\frac{A c_1}{\rho}+c_3)(\rho+x\rho_{x}+t\rho_{t})$\\ \hline \hline
\multirow{3}{*} & $\Upsilon^{x}=\frac{D \left((\rho + u)c_1 + c_2 \right) u_{tx}}{\rho} +\left( -\frac{c_1 D u_{x}}{\rho}+\frac{D (c_1 u+ c_2) \rho_{x}}{\rho^2}-\left( c_1 u - \frac{A c_1}{\rho} + c_3 \right) \rho - \left( \left( \rho + u \right) c_1 + c_2 \right) u \right)u_{t}$\\
$\mathscr{S}_{2}=\partial_{t}$ & $- \rho_{t}\left( \left( c_1 u+ c_3 \right) u + \frac{\left(c_1 \rho+ c_2 \right) A}{\rho} \right)$\\ 
& $\Upsilon^{t}=-(c_1(\rho+u)+c_2)u_{t}-(c_1 u-\frac{A c_1}{\rho}+c_3)\rho_{t}$\\ \hline \hline
\multirow{3}{*} & $\Upsilon^{x}=\frac{D  \left((\rho + u)c_1 + c_2 \right) t u_{xx}}{\rho} -\left( -\frac{c_1 D u_{x}}{\rho}+\frac{D (c_1 u+ c_2) \rho_{x}}{\rho^2}+\left( c_1 u - \frac{A c_1}{\rho} + c_3 \right) \rho + \left( \left( \rho + u \right) c_1 + c_2 \right) u \right)(1-t u_{x})$\\
$\mathscr{S}_{3}=t\partial_{x}+\partial_{u}$ & $- t \rho_{x}\left( \left( c_1 u+ c_3 \right) u + \frac{\left(c_1 \rho+ c_2 \right) A}{\rho} \right)$\\
& $\Upsilon^{t}=(c_1(\rho+u)+c_2)(1-t u_{x})-(c_1 u-\frac{A c_1}{\rho}+c_3)t \rho_{x}$\\ \hline \hline
\multirow{3}{*} & $\Upsilon^{x}=\frac{D \left((\rho + u)c_1 + c_2 \right) u_{xx}}{\rho} +\left( -\frac{c_1 D u_{x}}{\rho}+\frac{D (c_1 u+ c_2) \rho_{x}}{\rho^2}-\left( c_1 u - \frac{A c_1}{\rho} + c_3 \right) \rho - \left( \left( \rho + u \right) c_1 + c_2 \right) u \right)u_{x}$\\
$\mathscr{S}_{4}=\partial_{x}$ & $- \rho_{x}\left( \left( c_1 u+ c_3 \right) u + \frac{\left(c_1 \rho+ c_2 \right) A}{\rho} \right)$\\
& $\Upsilon^{t}=-(c_1(\rho+u)+c_2)u_{x}-(c_1 u-\frac{A c_1}{\rho}+c_3)\rho_{x}$\\ \hline \hline
\end{tabular}%
}
\caption{Table for Conserved Vectors.}
\label{CVT}
\end{table}
The conserved vectors $\Upsilon^{x}, \Upsilon^{t}$ are non-trivial for the vectors $\mathscr{S}_i$ for $i=1,2,3,4$.
\section{Results and Discussion}
\begin{enumerate}
    \item 
\begin{figure}[H]
    \centering
    \subfloat[3D graph of $\rho$.]{\includegraphics[width=0.45\textwidth]{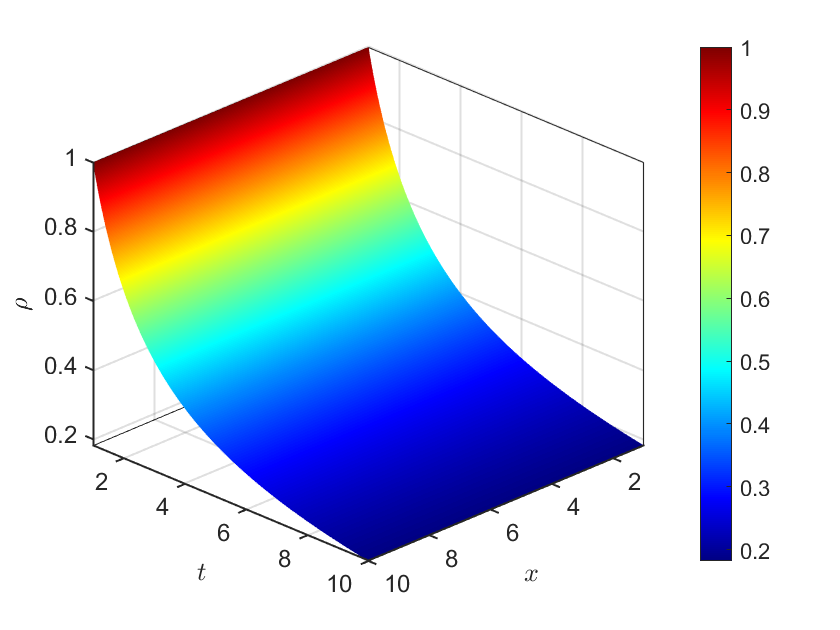}\label{fig:sub_TM1b}}
    \quad
    \subfloat[3D graph of $u$.]{\includegraphics[width=0.45\textwidth]{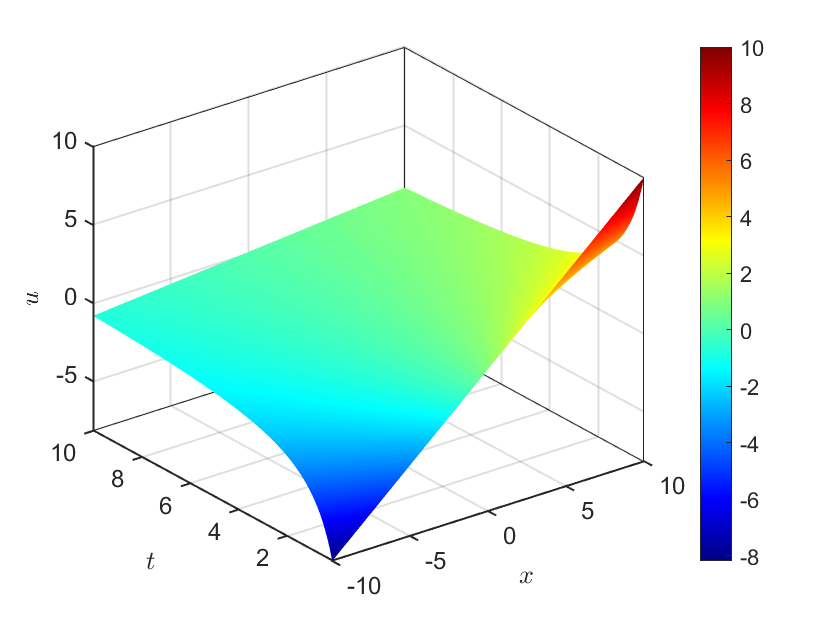}\label{fig:sub_TM1a}}
    \caption{Solution profiles of density $\rho$ and velocity $u$ at $p_1=1$, $p_2=2$ and $b=1$ for subalgebra $\mathcal{T}_{1}$ from Table \ref{OST}.}
    \label{fig:solution_profile}
\end{figure}
\begin{enumerate}[label=(\alph*)]
\item Figure \ref{fig:solution_profile}\subref{fig:sub_TM1b} represents the time-dependent decrease in vehicle density $(\rho)$, similar to the easing of traffic after rush hour. Initially, high density decreases over time, simulating the dispersion of vehicles and transitioning toward free-flow conditions, as one might observe on a highway network after peak usage. This pattern reflects the basic principle that, as time passes, an initially congested area will naturally spread out, reducing local density.
\item Figure \ref{fig:solution_profile}\subref{fig:sub_TM1a} illustrates a velocity profile $(u)$ that increases with position (x) but decreases with time $(t)$. This could represent a highway scenario where initial congestion near an entry point causes slower speeds. As vehicles move further from this point or as time progresses and congestion eases, speeds increase, mirroring the dynamics of traffic that gradually transitions from a bottleneck to a smoother, faster flow.
\end{enumerate}
\item 
\begin{figure}[H]
\centering
\subfloat[3D plot of $\rho$.]
   {\includegraphics[width=0.45\textwidth]{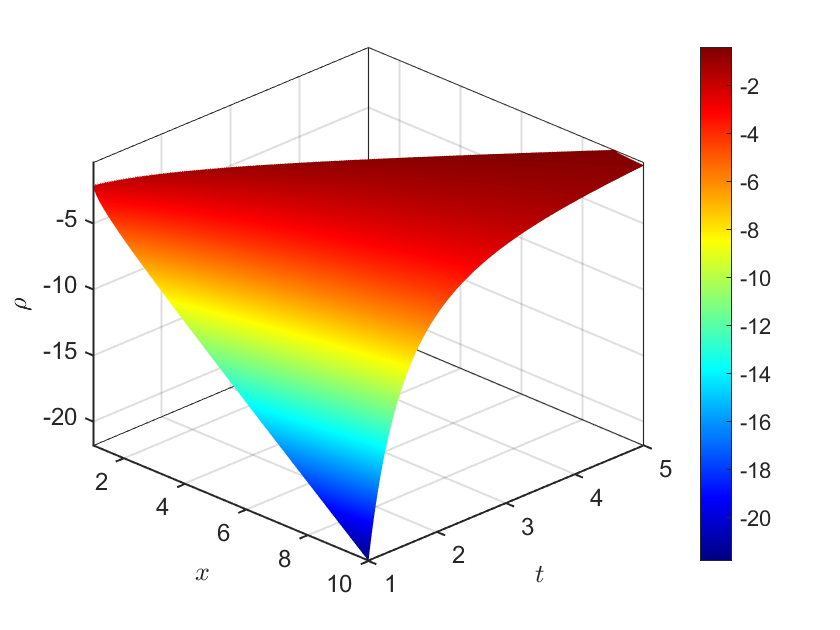}\label{fig:sub_TM2b}}
\quad
    \subfloat[3D plot of $u$.]
   {\includegraphics[width=0.45\textwidth]{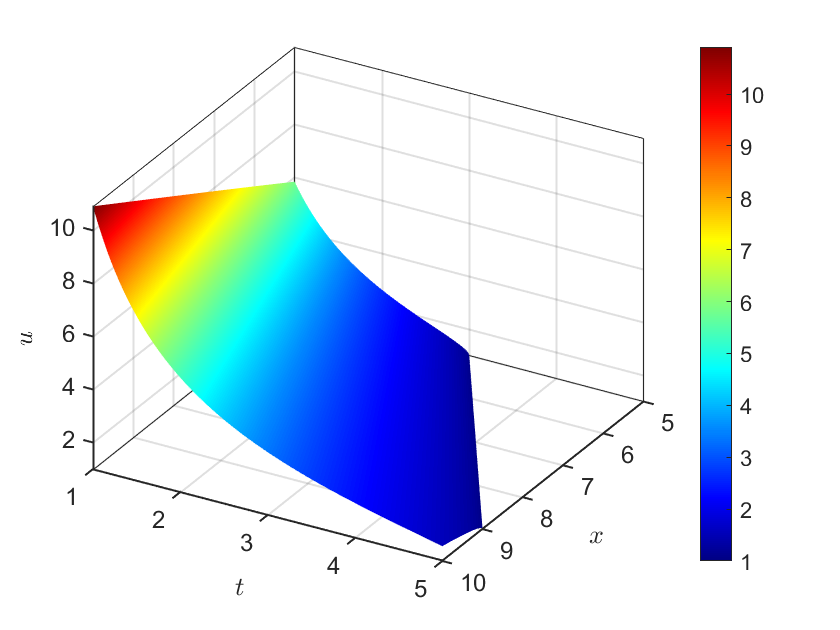}\label{fig:sub_TM2a}}
    \caption{Solution profiles of density $\rho$ and velocity $u$ at $b=1$, $A=1$ and $p_1=2$ for subalgebra $\mathcal{T}_{2}$ from Table \ref{OST}.}
    \label{F2}
\end{figure}
\begin{enumerate}[label=(\alph*)]
\item The graph \ref{F2}\subref{fig:sub_TM2b} illustrates a scenario in which the density decreases along the road $(x)$ and as time $(t)$ progresses. Initially, at the beginning of the road segment and at early times, the density of vehicles is high. However, as one moves further down the road and as time elapses, the density diminishes significantly. In real-world traffic, this could represent a situation where a large number of vehicles enter a highway. Then, they gradually disperse as they move toward their destinations. In addition, with time, the number of cars decreases on that road.
\item The graph \ref{F2}\subref{fig:sub_TM2a} shows an inverse relationship to the density plot. Initially, the speed is high at the beginning of the road and decreases dramatically over time and distance. This suggests that vehicles start at high speed. Then they quickly decelerate as they move further along the road. In real-world terms, this could represent cars entering a highway. They encounter increasing congestion, leading to reduced speeds and eventual slowdowns.
\end{enumerate}
\item 
\begin{figure}[H]   
\centering
 \subfloat[3D plot of $\rho$.]
{\includegraphics[width=0.45\textwidth]{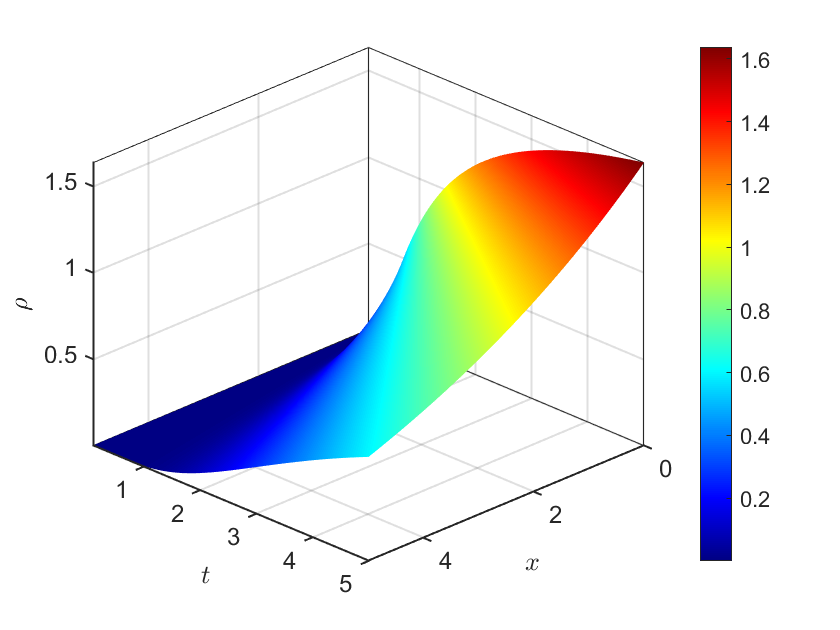}\label{fig:sub_TM3b}}
   \quad
     \subfloat[3D plot of $u$.]
{\includegraphics[width=0.45\textwidth]{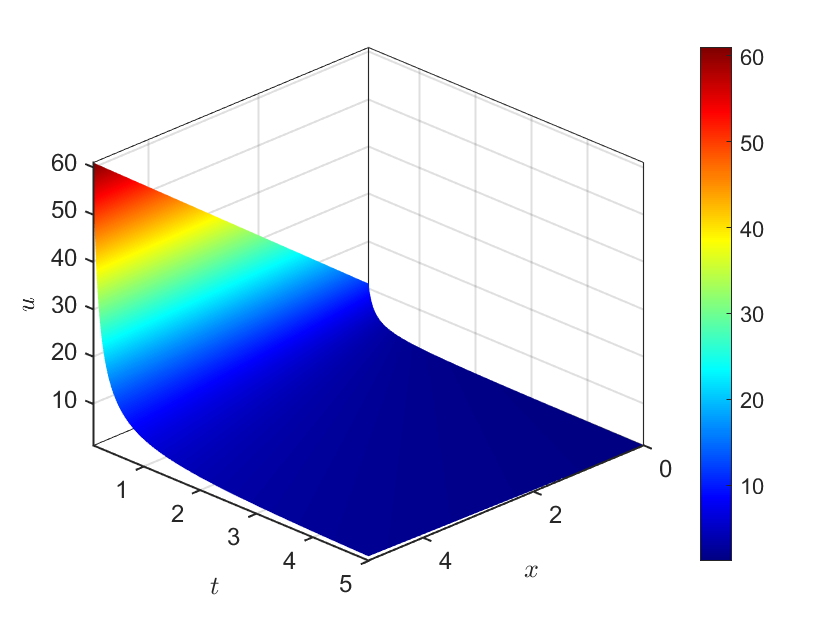}\label{fig:sub_TM3a}}
    \caption{Solution profiles of density $\rho$ and velocity $u$ at $p_1=2$, $A=1$ and $b=1$ for Eq. (\ref{E533}) .}
\end{figure}
\label{F3}
\begin{enumerate}[label=(\alph*)]
\item The graph \ref{F3}\subref{fig:sub_TM3b} illustrates how the density $\rho$ evolves over time $t$ and position $x$. Initially, the density is low near the origin, but it steadily increases as both $t$ and $x$ grow, forming a wave-like pattern. This behavior can be seen in real-life scenarios such as traffic jams, where the number of cars (density) increases in a region over time, or in fluid flow, where particles accumulate and the medium becomes denser as it moves forward.
\item The graph \ref{F3}\subref{fig:sub_TM3a} shows how the velocity $u$ changes with respect to time $t$ and position $x$. At the beginning (small $t$ and $x$), the velocity is very high, but it rapidly decreases and approaches a low, steady value as time and position increase. This pattern is similar to real-life traffic flow, where cars initially move quickly but slow down as congestion builds up over time and distance, or to a fluid that starts moving rapidly and then slows as it spreads out.
\end{enumerate}
\item
\begin{figure}[H]   
\centering
 \subfloat[3D plot of $\rho$.]
{\includegraphics[width=0.45\textwidth]{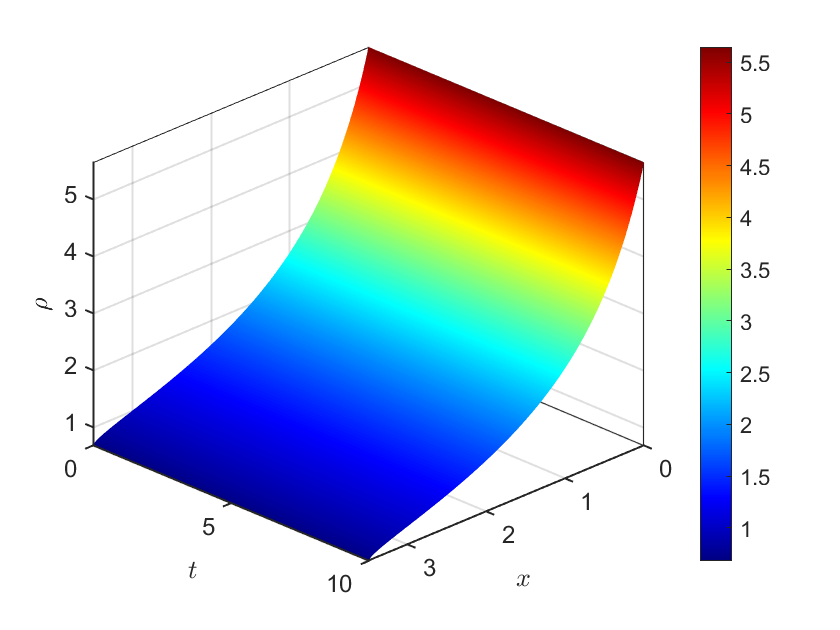}\label{fig:sub_TMP1b}}
   \quad
     \subfloat[3D plot of $u$.]
{\includegraphics[width=0.45\textwidth]{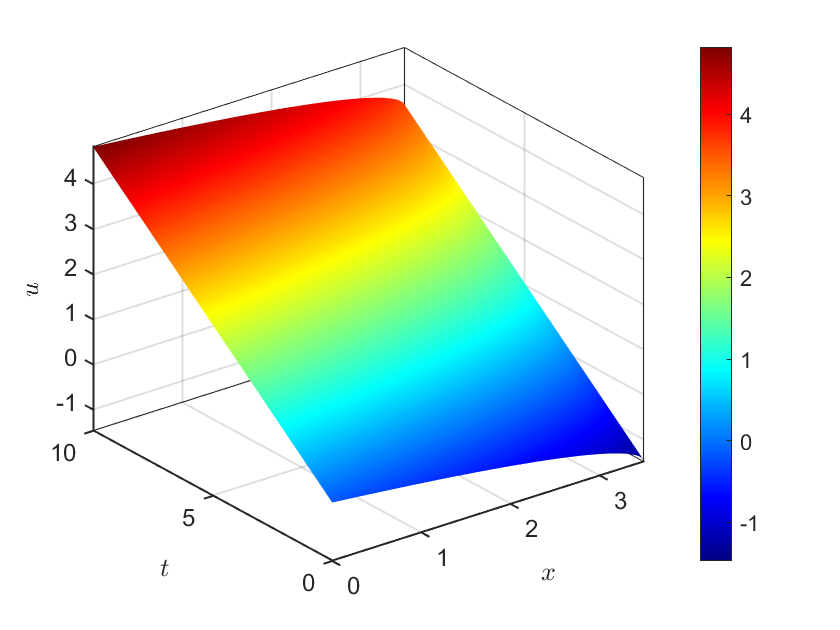}\label{fig:sub_TMP1a}}
    \caption{Solution profiles of density $\rho$ and velocity $u$ at $p_1=2$, $e_2=2$, $e_3=1, e_4=3$, and $p_2=1$ for Eq. (\ref{522}).}
    \label{F4}
\end{figure}
\begin{enumerate}[label=(\alph*)]
\item Figure \ref{F4}\subref{fig:sub_TMP1b} displays the vehicle density $(\rho)$ evolution in a non-viscous traffic flow model. The surface exhibits a gradient from high density (red) to low density (blue) across space and time. This pattern mirrors real-world scenarios where traffic congestion gradually dissipates after a bottleneck is removed, such as when an accident is cleared on a highway. The smooth transition from congested to free-flow conditions demonstrates the natural resolution of traffic density waves governed by the conservation principle in equation $\mathbb{R}_{1}$, providing valuable insights for traffic management strategies.
\item Figure \ref{F4}\subref{fig:sub_TMP1a} illustrates the velocity field $(u)$ corresponding to the density distribution, highlighting the inverse relationship between traffic density and speed. This reflects common stop-and-go traffic patterns where vehicles initially slow or stop when approaching congestion before accelerating as traffic thins. The model captures essential features of traffic waves observed on busy urban highways during rush hours, where speed fluctuations propagate upstream against the direction of travel, informing the development of adaptive traffic control systems.
 \end{enumerate}
 \item
 \begin{figure}[H]
    \begin{minipage}{.5\textwidth}
     \subfloat[3D plot of $\rho$]
   {\includegraphics[width=\textwidth]{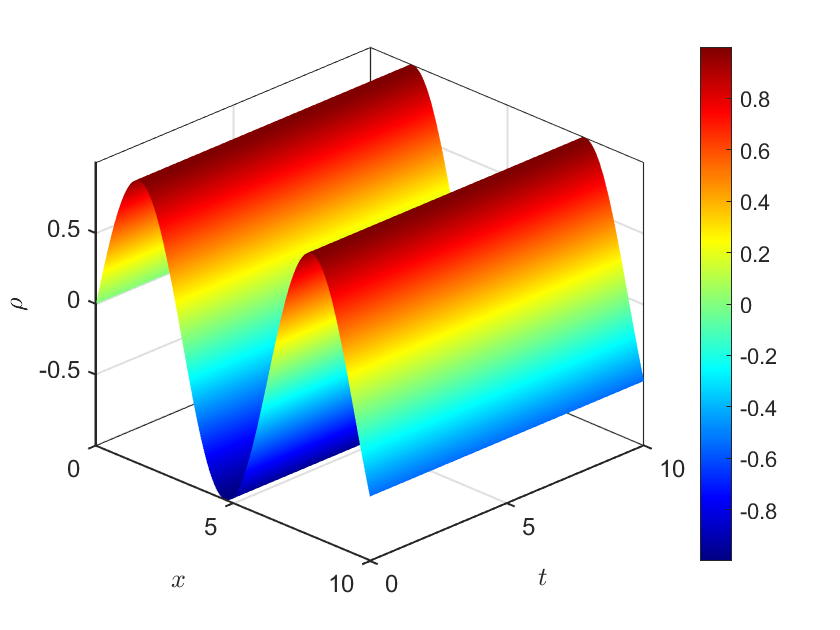}\label{CLp}}
\end{minipage}
\hfill    
\begin{minipage}{.5\textwidth}
 \subfloat[3D plot of $u$]
   {\includegraphics[width=\textwidth]{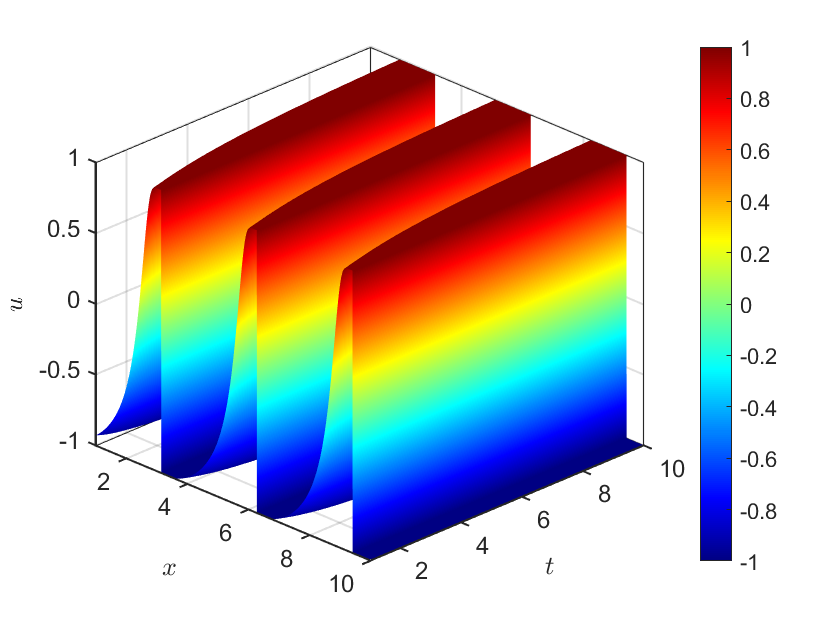}\label{Clu}}
\end{minipage}  
\caption{Plots of the exact solution 1 from Table \ref{CBES} for parameters $A=1$ and $c_1=1$.} 
\label{F9}
\end{figure}
\begin{enumerate}[label=(\alph*)]
\item Figure \ref{F9}\subref{CLp} displays the density profile, which oscillates smoothly between high (red) and low (blue) values, forming a periodic wave along the spatial axis. This represents a cnoidal wave soliton in density, modeling alternating clusters of high and low vehicle density-like a series of traffic jams and free-flowing sections moving along the highway. Such periodic density patterns are typical in real traffic, especially during rush hours or on ring roads with repeated congestion zones.
\item Figure \ref{F9}\subref{Clu} shows the velocity profile as a function of position and time. The sharp, periodic transitions from low (blue) to high (red) velocity form a repeating kink-like pattern along the road. This is a periodic kink (cnoidal) soliton, representing multiple moving fronts where traffic speed suddenly drops and then rises again. In real-life traffic, this model repeats stop-and-go waves, where vehicles repeatedly slow down and speed up, a common occurrence during heavy congestion or on circular roads.
\end{enumerate}
\item
 \begin{figure}[H]
    \begin{minipage}{.5\textwidth}
     \subfloat[3D plot of $\rho$]
   {\includegraphics[width=\textwidth]{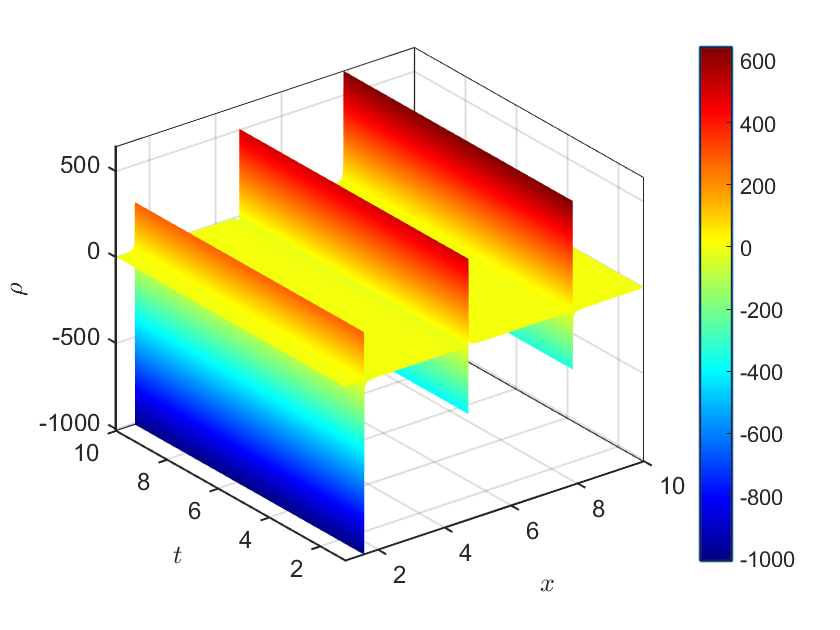}\label{CLp2}}
\end{minipage}
\hfill    
\begin{minipage}{.5\textwidth}
 \subfloat[3D plot of $u$]
   {\includegraphics[width=\textwidth]{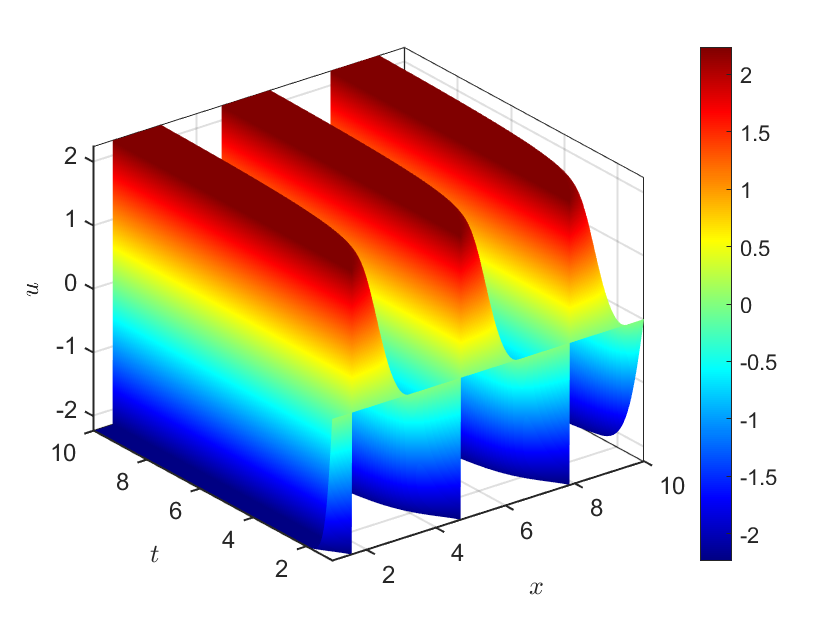}\label{Clu2}}
\end{minipage}  
\caption{Plots of the exact solution 2 from Table \ref{CBES} for parameters $A=5$ and $c_1=-1$.}
\label{F7}
\end{figure}
\begin{enumerate}[label=(\alph*)]
\item Figure \ref{F7}\subref{CLp2} displays the density as a periodic cnoidal soliton, with very sharp, repeating peaks and troughs. This represents recurring clusters of high and low vehicle density-like repeated traffic jams and clear stretches-moving along the highway. The soliton structure means these patterns are stable and persist over time, similar to real traffic waves during rush hour.
\item Figure \ref{F7}\subref{Clu2} shows the velocity as a periodic kink-type (cnoidal) soliton, with sharp transitions between high and low speeds repeating along the road. In real-life traffic, this model has multiple stop-and-go waves where cars repeatedly slow down and speed up, such as during heavy congestion or on ring roads.
\end{enumerate}
\item
 \begin{figure}[H]
    \begin{minipage}{.5\textwidth}
     \subfloat[3D plot of $\rho$]
   {\includegraphics[width=\textwidth]{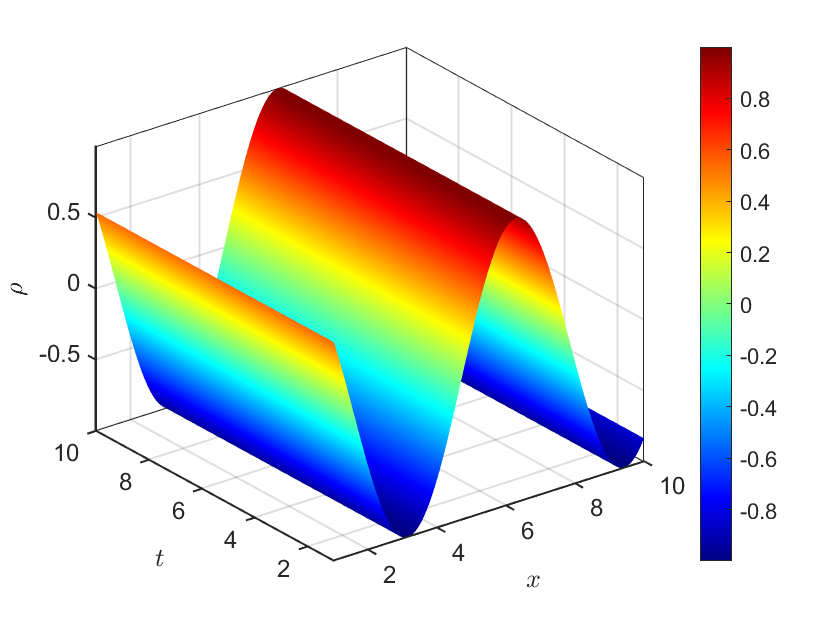}\label{CLp3}}
\end{minipage}
\hfill    
\begin{minipage}{.5\textwidth}
 \subfloat[3D plot of $u$]
   {\includegraphics[width=\textwidth]{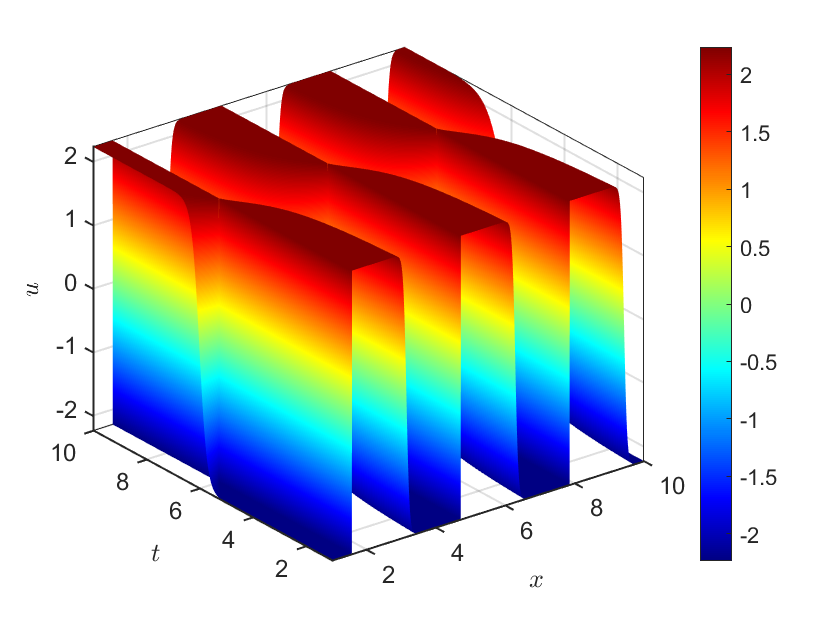}\label{Clu3}}
\end{minipage}  
\caption{Plots of the exact solution 3 from Table \ref{CBES} for parameters $A=5$ and $c_1=-6$.}
\label{F8}
\end{figure}
\begin{enumerate}[label=(\alph*)]
\item Figure \ref{F8}\subref{CLp3} shows that the density oscillates periodically between high (red) and low (blue) values, forming a wave-like pattern along the spatial axis. In real-life traffic, this could represent a series of repeating traffic jams and free-flowing regions along a highway, as if multiple congestion zones appear and move over time. The periodic nature suggests a cnoidal wave soliton, which is a type of soliton that repeats periodically rather than forming a single pulse or kink. Such patterns can occur during rush hours or on ring roads where traffic density naturally oscillates due to entry and exit flows.
\item Figure \ref{F8}\subref{Clu3} depicts the velocity profile, which shows sharp transitions between high (red) and low (blue) velocities, with these transitions repeating periodically along the spatial axis. This pattern reflects the presence of multiple moving fronts where traffic speed suddenly drops and then rises again, corresponding to the boundaries between congested and free-flowing regions. In real traffic, this could model several stop-and-go waves propagating through the flow, as often observed on busy highways with repeated bottlenecks. The sharp, stable transitions are characteristic of kink-type solitons, but the periodic repetition across space suggests a periodic kink or cnoidal soliton structure.
\end{enumerate}
\item
 \begin{figure}[H]
    \begin{minipage}{.5\textwidth}
     \subfloat[3D plot of $\rho$]
   {\includegraphics[width=\textwidth]{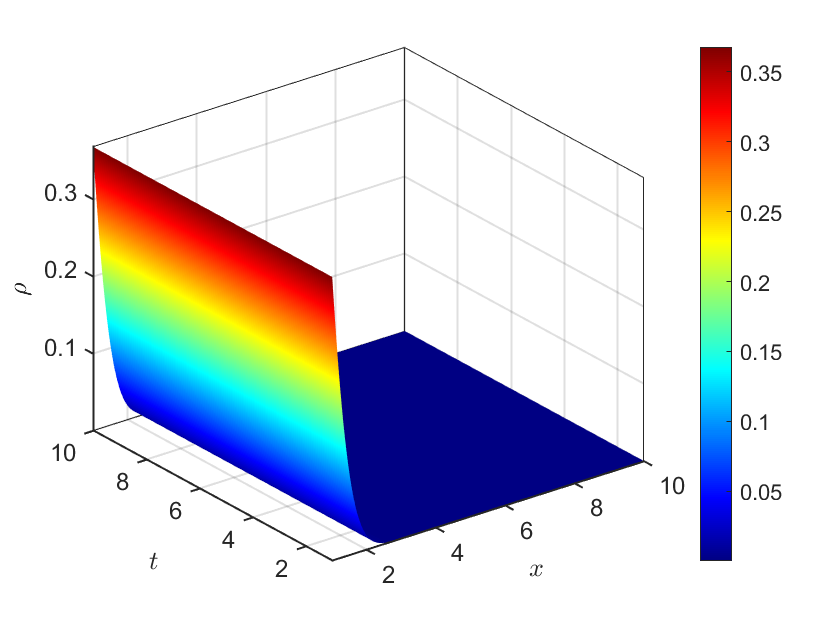}\label{CLp4}}
\end{minipage}
\hfill    
\begin{minipage}{.5\textwidth}
 \subfloat[3D plot of $u$]
   {\includegraphics[width=\textwidth]{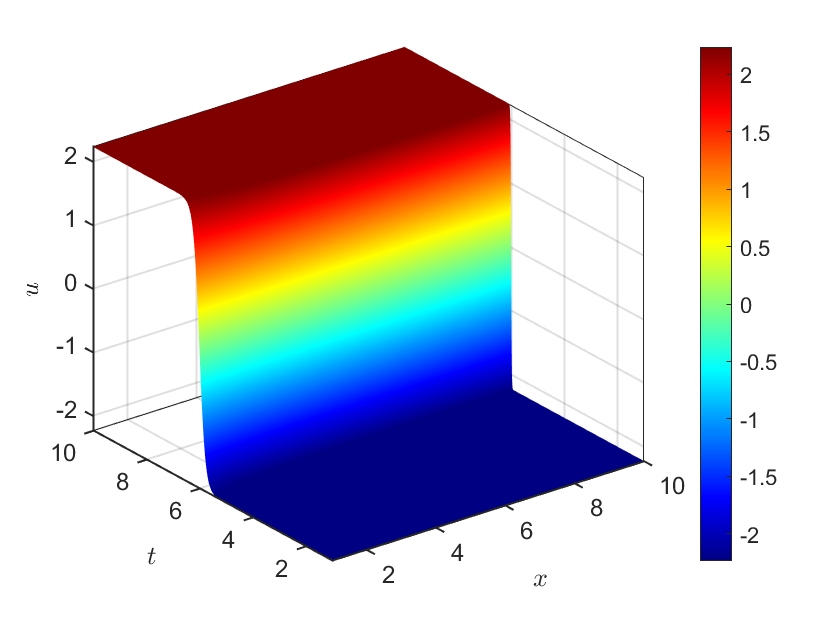}\label{Clu4}}
\end{minipage}  
\caption{Plots of the exact solution 4 from Table \ref{CBES} for parameters $A=5$ and $c_1=-6$.}
\label{F10}
\end{figure}
\begin{enumerate}[label=(\alph*)]
\item Figure \ref{F10}\subref{CLp4} shows the density $\rho(x,t)$, also as a kink-type soliton. The high-density region (red) moving through space and time represents a cluster of vehicle traffic jams traveling along the highway. As time progresses, this dense cluster moves, just like how actual traffic jams can propagate backward along a road. The soliton structure ensures this jam remains stable and self-organized as it travels, mirroring real-life congestion waves.
\item Figure \ref{F10}\subref{Clu4} shows the velocity profile $u(x,t)$ as a kink-type soliton. The sharp transition from low (blue) to high (red) velocity across space and time represents a sudden change in traffic speed. In real life, this model shows how a traffic jam suddenly clears and cars quickly accelerate. The soliton nature means this wave maintains its shape as it moves, just like persistent stop-and-go waves observed in real traffic.
\end{enumerate}
\end{enumerate}
\section{Weak Discontinuity Structures}
The equation \eqref{TM} can be expressed in matrix form as
\begin{equation}
    H_{t}+B H_{x}+ C H_{xx}= F,
\end{equation}
where $H=(\rho, u)^{T}$ and $F=(0,0)^{T}$ are column vectors, with the superscript $T$ indicating transposition, and
\begin{alignat*}{2}
B = &
\left[
\begin{array}{cc}
u & \rho \\
\frac{A}{\rho} & u
\end{array}
\right]\text{,} \qquad
C = &
\left[
\begin{array}{cc}
0 & 0\\
0 & \frac{-D}{\rho}
\end{array}
\right]\text{.}
\end{alignat*}
We will take $D=0$ for further calculation 
The characteristic roots associated with the matrix $B$ are 
\begin{equation}\label{862}
    \lambda_{1}= u-\sqrt{A} \hspace{0.2cm} \text{and} \hspace{0.2cm} \lambda_{2}=u+\sqrt{A}\text{.}
\end{equation}
The associated left and right characteristic eigenvectors are provided by
\begin{equation}\label{863}
\begin{split}
&l_{1}=(\frac{-\sqrt{A}}{\rho}, 1), \quad  r_{1}=(\frac{-\rho}{\sqrt{A}}, 1)^{T}, \quad l_{2}=(\frac{\sqrt{A}}{\rho}, 1), \quad  r_{2}=(\frac{\rho}{\sqrt{A}}, 1)^{T}\text{.}
    \end{split}
\end{equation}
\subsection{\texorpdfstring{$C^{1}$}{C1}-Wave}
\qquad Let us consider the $C^{1}$-Wave also referred to as the weak discontinuity, arising from $(x_0, t_0)$ and propagating along the fastest characteristic defined by $\frac{d x}{d t}=\lambda_2$, which precedes the shock $\mathcal{S}$. The propagation of weak discontinuity can be effectively analyzed using the following transport equation \cite{jeffrey1976quasilinear,sharma2025analysis}:
\begin{equation}\label{864}
    l_2\left(\frac{d K}{d t}+(H_{x}+K)(\nabla\lambda_2) K\right)+((\nabla l_2) K)^{T} \frac{d H}{d t}+(l_2 K)((\nabla \lambda_2)H_{x}+(\lambda_2)_{x})=0\text{,}
\end{equation}
where $K=\pi(t) r_2$ represents the discontinuity in $H_{x}$ along the weak discontinuity curve, where $\pi$ denotes the wave amplitude, $\nabla=(\frac{\partial}{\partial \rho}, \frac{\partial}{\partial u})$, and $\frac{d x}{d t}=\lambda_2$. The subsequent first-order linear equation is derived by simplifying Eq.
\eqref{864} using \eqref{862} and \eqref{863}
\begin{equation}\label{E865}
    \frac{d \pi}{d t}+\pi^2+\frac{\pi}{2}\left(\frac{\rho_{x} \sqrt{A}}{\rho}+5 u_{x}\right)=0\text{.}
\end{equation}
Using the solution of subalgebra $\mathcal{T}_1$ from Table \ref{OST}, we get the following equation as follows:
\begin{equation}\label{E865_2}
    \frac{d \pi}{d t}+\pi^2+\Psi(x,t)\pi=0, \hspace{0.2cm} \frac{d x}{d t}=\frac{x+p_1}{t+b}+\sqrt{A}\text{,}
\end{equation}
where 
\begin{equation}
    \Psi(x,t)=\frac{5}{2 (t+b)}\text{.}
\end{equation}
The solution of Equation (\ref{E865_2}) can be expressed in quadrature form:
\begin{equation}
\Pi(t) = \frac{E(t)}{1 + \Pi_0 F(t)} \hspace{0.1cm}\text{,} 
\end{equation}
where \( E(t) \) and \( F(t) \) are defined as follows:
\[
E(t) = \frac{1}{2}\exp\left( \int_{1}^{t} -\Psi(x(z), z) \, dz \right),
\quad
F(t) = \int_{1}^{t} \exp\left( \int_{1}^{z} -\Psi(x(y), y) \, dy \right) dz\text{,}
\]
and the term $\Pi_0$ signifies the initial amplitude. Examining the integrals $E(t)$ and $F(t)$ indicates that $E(t)$ approaches zero and $F(t)$ remains finite as $t$ approaches infinity. The solutions presented in the subalgebra $\mathcal{T}_{1}$ from Table \ref{OST} for $b=1$ and $A=1$ demonstrate that Figure \ref{LF}\subref{fig:sub_TMswa} shows the expansive wave behavior for $\Pi_0 > 0$, indicating that the wave finally diminishes and disappears after a finite duration. Furthermore, it is seen that the wave decreases more quickly as $\rho$ increases. In contrast, when $\Pi_0 < 0$ and $|\Pi_0| < \Pi_c$, where $\Pi_c = \lim\limits_{t \to \infty} \frac{1}{F(t)}$ is designated as the critical value, the wave persists as an expansive wave that ultimately diminishes and disappears, as illustrated in Figure \ref{LF}\subref{fig:sub_TMswb}. However, suppose $|\Pi_0|\ge \Pi_c$, that is, the initial discontinuity wave exceeds the critical threshold. In that case, the wave transforms into a shock after a finite duration, resulting in a compressive wave, as illustrated in Figure \ref{LF}\subref{fig:sub_TMswb}. A reduction in the factor extends the duration of shock production, signifying the accumulation of a substantial volume of fluid behind the frozen flow.
 \begin{figure}[H]
    \begin{minipage}{.49\textwidth}
     \subfloat[\footnotesize {}]
   {\includegraphics[width=\textwidth]{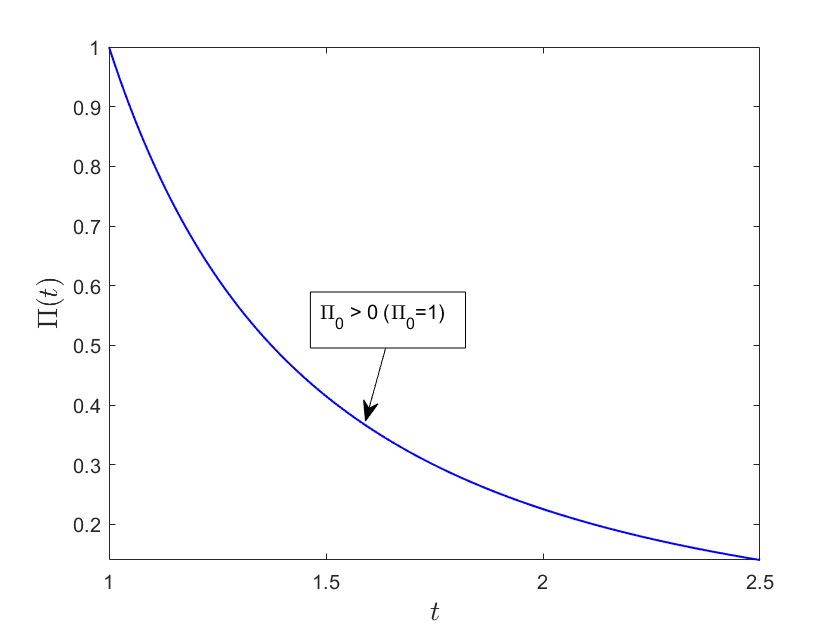}\label{fig:sub_TMswa}}
\end{minipage}
\hfill    
\begin{minipage}{.49\textwidth}
 \subfloat[\footnotesize {}]
   {\includegraphics[width=\textwidth]{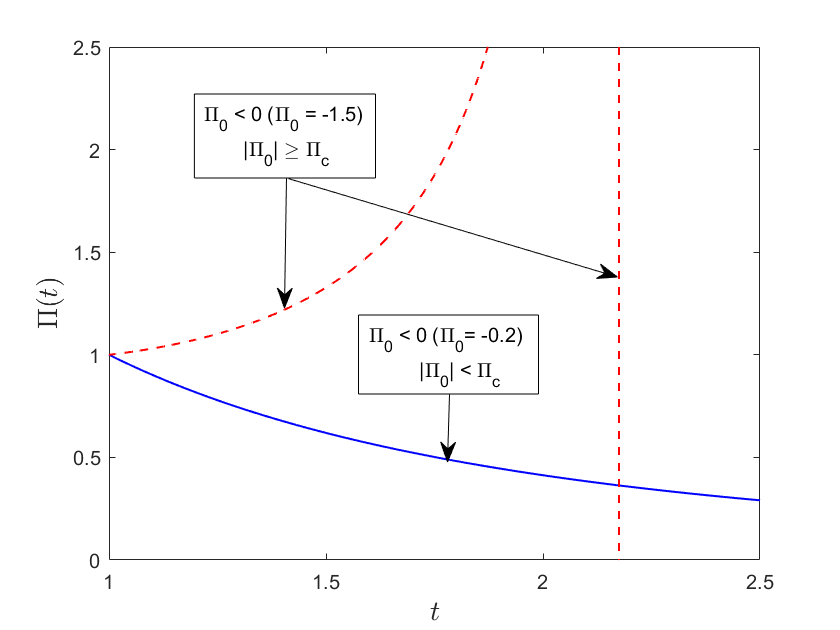}\label{fig:sub_TMswb}}
\end{minipage}  
\caption{Variation of $\Pi(t)$ with $b=1$ and $A=1$.}
\label{LF}
\end{figure}
\section{Conclusions}
\qquad This study successfully employs Lie symmetry analysis to a macroscopic traffic flow model, identifying infinitesimal generators and constructing an optimal system of one-dimensional subalgebras for effective symmetry reductions. Combining the creation of exact solutions with the direct technique of obtaining conservation laws ensures physically realistic findings. Analysis of the non-linear self-adjointness property strengthens the physical interpretation of these conservation rules. This work generates solitonic solutions, including parabola, kink-type, and peakon-type solitons, by applying traveling wave transformations, providing an understanding of nonlinear wave dynamics relevant to traffic evolution. Consider, for example, stop-and-go traffic patterns commonly observed during rush hour; these can be mathematically represented and better understood through these solitonic solutions (kink-type soliton). The impacts on traffic behavior of left-invariant solutions inside the given symmetry framework are clarified using their illumination of certain features of shock wave production and flow stability. The model allows us to analyze the impact of infrastructure changes on traffic. In essence, this work enhances the mathematical understanding of traffic dynamics, providing a solid theoretical basis applicable to practical scenarios. This foundation can guide actual congestion control plans, like traffic signal location and timing optimization, or highway lane addition or removal evaluation.
\subsection{Potential Application}
\qquad The exact solutions and solitonic wave structures obtained in this study have direct implications for traffic management. Kink, peakon, and parabola forms of solitonic waves help to simulate the development and spread of localized traffic congestion. Understanding these patterns can help in designing dynamic speed limits, ramp metering, and other control strategies to dissipate traffic jams before they escalate. The invariant and exact solutions provide benchmarks for real-time traffic monitoring and can inform the development of feedback control laws. These can be used to adjust signal timings, ramp gates, and connected vehicle behavior to maintain optimal traffic flow. Furthermore, supporting applications in Intelligent Transportation Systems (ITS) are the results, which enhance predictive traffic models, and data-driven control helps to manage congestion. Thus, the mathematical insights presented contribute to real-time congestion mitigation, travel time reliability, and more sustainable traffic systems.
\section*{Future Recommendation}
Future work should start with thorough empirical validation by calibrating the model with several real-world traffic datasets. To enhance model realism, it's essential to incorporate complex, external factors such as weather conditions, traffic incidents, and variations in driver behavior. Exploring higher-order traffic flow models that account for driver anticipation and differing vehicle types is also recommended. The investigation should be on the use of derived symmetries in creating real-time traffic control systems, such as ramp metering or varying speed limits. Finally, efforts should be directed at developing efficient numerical schemes for real-time applications, extending the analysis to include whole traffic networks, integrating machine learning techniques for improved predictive capability, and evaluating the environmental impact of traffic congestion for the best sustainable management.
\section*{Funding declaration}
The first author acknowledges financial support from the Ministry of Education (MoE). Aniruddha Kumar Sharma, the second author, expresses gratitude for the financial assistance provided by the Council of Scientific and Industrial Research (CSIR) India, Grant/Award Number: 09/0143(12918)/2021-EMR-I.

\section*{Author Contributions}
\quad \quad \textbf{Urvashi Joshi}: Led the conceptualization and design of the study, carried out the formal analysis and investigation, developed and implemented the software, performed data visualization, and prepared both the original manuscript draft and subsequent revisions.

\textbf{Aniruddha Kumar Sharma}: Provided methodological guidance, contributed to supervision, and assisted with manuscript review and editing.  

\textbf{Rajan Arora}: Contributing to the review and refinement of the manuscript.

\section*{Conflict of Interest}
The author's report states that there is no conflict of interest associated with this publication.
\section*{Data Availability Statement}
No data used in this manuscript.
\bibliographystyle{elsarticle-num}
\bibliography{ref.bib}

\begin{thebibliography}{10}
\expandafter\ifx\csname url\endcsname\relax
  \def\url#1{\texttt{#1}}\fi
\expandafter\ifx\csname urlprefix\endcsname\relax\def\urlprefix{URL }\fi
\expandafter\ifx\csname href\endcsname\relax
  \def\href#1#2{#2} \def\path#1{#1}\fi

\bibitem{horvat2015traffic}
R.~Horvat, G.~Kos, M.~{\v{S}}evrovi{\'c}, Traffic flow modelling on the road network in the cities, Tehni{\v{c}}ki vjesnik 22~(2) (2015) 475--486.

\bibitem{mohan2013state}
R.~Mohan, G.~Ramadurai, State-of-the art of macroscopic traffic flow modelling, International Journal of Advances in Engineering Sciences and Applied Mathematics 5 (2013) 158--176.

\bibitem{kanagaraj2013evaluation}
V.~Kanagaraj, G.~Asaithambi, C.~N. Kumar, K.~K. Srinivasan, R.~Sivanandan, Evaluation of different vehicle following models under mixed traffic conditions, Procedia-Social and Behavioral Sciences 104 (2013) 390--401.

\bibitem{van2015genealogy}
F.~van Wageningen-Kessels, H.~Van~Lint, K.~Vuik, S.~Hoogendoorn, Genealogy of traffic flow models, EURO Journal on Transportation and Logistics 4~(4) (2015) 445--473.

\bibitem{ferrara2018microscopic}
A.~Ferrara, S.~Sacone, S.~Siri, A.~Ferrara, S.~Sacone, S.~Siri, Microscopic and mesoscopic traffic models, Freeway traffic modelling and control (2018) 113--143.

\bibitem{aw2000resurrection}
A.~Aw, M.~Rascle, Resurrection of" second order" models of traffic flow, SIAM journal on applied mathematics 60~(3) (2000) 916--938.

\bibitem{aw2002derivation}
A.~Aw, A.~Klar, M.~Rascle, T.~Materne, Derivation of continuum traffic flow models from microscopic follow-the-leader models, SIAM Journal on applied mathematics 63~(1) (2002) 259--278.

\bibitem{richards1956shock}
P.~I. Richards, Shock waves on the highway, Operations research 4~(1) (1956) 42--51.

\bibitem{lighthill1955kinematic}
M.~J. Lighthill, G.~B. Whitham, On kinematic waves ii. a theory of traffic flow on long crowded roads, Proceedings of the royal society of london. series a. mathematical and physical sciences 229~(1178) (1955) 317--345.

\bibitem{paliathanasis2022lie}
A.~Paliathanasis, P.~G. Leach, Lie symmetry analysis of the aw--rascle--zhang model for traffic state estimation, Mathematics 11~(1) (2022) 81.

\bibitem{paliathanasis2023symmetry}
A.~Paliathanasis, Symmetry analysis for the 2d aw-rascle traffic-flow model of multi-lane motorways in the euler and lagrange variables, Symmetry 15~(8) (2023) 1525.

\bibitem{bluman2010applications}
G.~W. Bluman, Applications of symmetry methods to partial differential equations, Springer, 2010.

\bibitem{olver1993applications}
P.~J. Olver, Applications of Lie groups to differential equations, Vol. 107, Springer Science \& Business Media, 1993.

\bibitem{hussain2024lie}
A.~Hussain, M.~Usman, F.~Zaman, Lie group analysis, solitons, self-adjointness and conservation laws of the nonlinear elastic structural element equation, Journal of Taibah University for Science 18~(1) (2024) 2294554.

\bibitem{khoudari2024microscopic}
N.~Khoudari, From Microscopic to Macroscopic Scales: Traffic Waves and Sparse Control, Temple University, 2024.

\bibitem{shagolshem2024exact}
S.~Shagolshem, B.~Bira, K.~Nagaraja, Exact solutions, conservation laws, and shock wave propagation of two-lanes traffic flow model via lie symmetry, Physics of Fluids 36~(8) (2024).

\bibitem{goard2014symmetries}
J.~Goard, S.~Al-Nassar, Symmetries for initial value problems, Applied Mathematics Letters 28 (2014) 56--59.

\bibitem{al2012nonclassical}
S.~K. Al-Nassar, Nonclassical symmetry analysis of second order parabolic partial differential equations (2012).

\bibitem{zhang2010classical}
Z.~Zhang, Y.~Chen, Classical and nonclassical symmetries analysis for initial value problems, Physics Letters A 374~(9) (2010) 1117--1120.

\bibitem{sharma2025analysis}
A.~K. Sharma, S.~Shagolshem, R.~Arora, Analysis of wave propagation and conservation laws for a shallow water model with two velocities via lie symmetry, Communications in Nonlinear Science and Numerical Simulation (2025) 108637.

\bibitem{hu2015direct}
X.~Hu, Y.~Li, Y.~Chen, A direct algorithm of one-dimensional optimal system for the group invariant solutions, Journal of Mathematical Physics 56~(5) (2015).

\bibitem{avdonina2013exact}
E.~D. Avdonina, N.~H. Ibragimov, R.~Khamitova, Exact solutions of gasdynamic equations obtained by the method of conservation laws, Communications in Nonlinear Science and Numerical Simulation 18~(9) (2013) 2359--2366.

\bibitem{ibragimov2006integrating}
N.~H. Ibragimov, Integrating factors, adjoint equations and lagrangians, Journal of Mathematical Analysis and Applications 318~(2) (2006) 742--757.

\bibitem{ibragimov2007new}
N.~H. Ibragimov, A new conservation theorem, Journal of Mathematical Analysis and Applications 333~(1) (2007) 311--328.

\bibitem{ibragimov2011nonlinear}
N.~H. Ibragimov, Nonlinear self-adjointness and conservation laws, Journal of Physics A: Mathematical and Theoretical 44~(43) (2011) 432002.

\bibitem{sil2020nonlocal}
S.~Sil, T.~R. Sekhar, D.~Zeidan, Nonlocal conservation laws, nonlocal symmetries and exact solutions of an integrable soliton equation, Chaos, Solitons \& Fractals 139 (2020) 110010.

\bibitem{liu2020existence}
M.~Liu, H.~Dong, On the existence of solution, lie symmetry analysis and conservation law of magnetohydrodynamic equations, Communications in Nonlinear Science and Numerical Simulation 87 (2020) 105277.

\bibitem{shagolshem2023classification}
S.~Shagolshem, B.~Bira, Classification of nonlocal symmetries and exact solutions for 3$\times$ 3 chaplygin gas equation with conservation laws, Physics of Fluids 35~(5) (2023).

\bibitem{friedrich2022conservation}
J.~Friedrich, S.~G{\"o}ttlich, A.~Keimer, L.~Pflug, Conservation laws with nonlocality in density and velocity and their applicability in traffic flow modelling, in: XVI International Conference on Hyperbolic Problems: Theory, Numerics, Applications, Springer, 2022, pp. 347--357.

\bibitem{bluman2010local}
G.~W. Bluman, A.~F. Cheviakov, S.~C. Anco, G.~W. Bluman, A.~F. Cheviakov, S.~C. Anco, Local transformations and conservation laws, Applications of Symmetry Methods to Partial Differential Equations (2010) 1--120.

\bibitem{jeffrey1976quasilinear}
A.~Jeffrey, Quasilinear hyperbolic systems and waves,\hspace{0.05cm} Pitman Publishing, London (1976).

\end{thebibliography}
\end{document}